\documentclass[12pt]{amsart}
\usepackage{amssymb}
\usepackage{amsfonts}
\usepackage{amscd}
\usepackage[T1]{fontenc}

\swapnumbers
\let\Bbb\mathbb

\newcommand{\Q}{\mathbb{Q}}
\newcommand{\F}{\mathbb{F}}
\newcommand{\Z}{\mathbb{Z}}
\newcommand{\N}{\mathbb{N}}

\def\11{{\mathbf 1}}

\def\A{\mathbb{A}}
\def\B{\mathbb{B}}
\def\Q{\mathbb{Q}}

\def\F{\mathbb{F}}

\def\cL{{\mathcal L}}
\def\cM{{\mathcal M}}

\def\cO{{\mathcal O}}

\mathchardef\alphag="7C0B
\mathchardef\betag="7C0C
\mathchardef\gammag="7C0D
\mathchardef\deltag="7C0E
\mathchardef\varepsilong="7C22
\mathchardef\varphig="7C27
\mathchardef\psig="7C20
\mathchardef\zetag="7C10
\mathchardef\epsilong="7C0F
\mathchardef\rhog="7C1A
\mathchardef\taug="7C1C
\mathchardef\upsilong="7C1D
\mathchardef\iotag="7C13
\mathchardef\thetag="7C12
\mathchardef\pig="7C19
\mathchardef\sigmag="7C1B
\mathchardef\etag="7C11
\mathchardef\omegag="7C21
\mathchardef\kappag="7C14
\mathchardef\lambdag="7C15
\mathchardef\mug="7C16
\mathchardef\xig="7C18
\mathchardef\chig="7C1F
\mathchardef\nug="7C17
\mathchardef\varthetag="7C23
\mathchardef\varpig="7C24
\mathchardef\varrhog="7C25
\mathchardef\varsigmag="7C26
\mathchardef\Omegag="7C0A
\mathchardef\Thetag="7C02
\mathchardef\Sigmag="7C06
\mathchardef\Deltag="7C01
\mathchardef\Phig="7C08
\mathchardef\Gammag="7C00
\mathchardef\Psig="7C09
\mathchardef\Lambdag="7C03
\mathchardef\Xig="7C04
\mathchardef\Pig="7C05
\mathchardef\Upsilong="7C07

\newtheorem{thm}{Theorem}
\newtheorem{lem}{Lemma}

\newtheorem{cor}[thm]{Corollary}

\newtheorem{remark}{Remark}
\newtheorem{ex}{Example}
\newtheorem{note}{Note}

\newtheorem{claim}{Claim}
\newtheorem{definition*}{Definition}
\newtheorem*{cor*}{Corollary}

\theoremstyle{plain}

\numberwithin{equation}{subsection}

\def\boxit#1#2{\setbox1=\hbox{\kern#1{#2}\kern#1}%
\dimen1=\ht1 \advance\dimen1 by #1
\dimen2=\dp1 \advance\dimen2 by #1
\setbox1=\hbox{\vrule height\dimen1 depth\dimen2\box1\vrule}%
\setbox1=\vbox{\hrule\box1\hrule}%
\advance\dimen1 by .4pt \ht1=\dimen1
\advance\dimen2 by .4pt \dp1=\dimen2 \box1\relax}

  {\par\medskip\noindent #1\par\begingroup%
    \advance\leftskip by 1em\advance\rightskip by 1em}%
  {\par\endgroup}

\begin{document}

\setcounter{tocdepth}{2} % Show subsection in table of contents

\title[Some Supplements to Feferman-Vaught]
{Some Supplements to Feferman-Vaught related to the Model Theory of Adeles} % $\ZZ_p$ in $\QQ_p$}

\author{Jamshid Derakhshan}
\address{University of Oxford, Mathematical Institute,
24-29 St Giles', Oxford OX1 3LB, UK}
\email{derakhsh@maths.ox.ac.uk}

\author{Angus Macintyre}
\address{Queen Mary, University of London,
School of Mathematical Sciences, Queen Mary, University of London, Mile End Road, London E1 4NS, UK}
\email{angus@eecs.qmul.ac.uk}

\subjclass[2010]{Primary 03C10,03C40,03C60,03C95,03C98; Secondary 11U05,11U09,12L05,12L12,12E50}

\keywords{Definability, Restricted Products and Adeles, Quantifier Elimination, Feferman-Vaught Theorems, 
Hyperrings, Valued Fields}

\begin{abstract} We give foundational results for the model theory of $\A_K^{fin}$, the ring of finite adeles over a number field, 
construed as a restricted product of local fields. In contrast to Weispfenning we work in the language of ring theory, and various 
sortings interpretable therein. In particular we give a systematic treatment of the product valuation and the valuation monoid. 
Deeper results are given for the adelic version of Krasner's hyperfields, relating them to the Basarab-Kuhlmann formalism.
\end{abstract}

\maketitle

\tableofcontents

\section{Introduction}\label{section-introduction}
We have recently revisited Weispfenning's work \cite{weispfenning-hab} on the rings of adeles $\A_K$ over number fields $K$. 
That work in turn depends on the classic paper of Feferman-Vaught \cite{FV} on generalized products. 
Our objective is to obtain the most refined analysis possible of definable sets in $\A_K$ (paying special attention to 
uniformity in $K$). One intended application is to computation of measures and integrals over $\A_K$. A first 
paper \cite{DM-ad} on this will soon be available. 
We think of our approach as rather more geometric, and less abstractly model theoretic, than the analysis in \cite{weispfenning-hab} 
and \cite{FV}. We prefer to work in the language of ring theory (or sometimes topological ring theory), without the 
Boolean or lattice-theoretic scaffolding from \cite{weispfenning-hab} and \cite{FV} (which has much more general applicability). We 
wish to stress that we add little to the foundations of the theory of generalized products. The treatment in \cite{weispfenning-hab} and 
\cite{FV} can hardly be improved. Rather, we work directly on the adeles $\A_K$ as a ring. However, we depend on various 
quantifier eliminations for completions $K_v$ and some of these are in many-sorted languages appropriate to Henselian fields, 
so we need a version of \cite{FV} for many-sorted structures. Moreover $\A_K$ is a restricted product in the sense of 
\cite{FV} even in a language (like ring-theory) with function-symbols. Though it is implicit in \cite{FV} how to deal with 
function symbols and sorts, we prefer to prepare this short paper providing foundations appropriate to the adelic setting. 
Further motivation is provided by the model theory of the product valuation on $\A_K$. The image is a submonoid of a 
lattice ordered group, and so not literally itself a restricted product. But some simple technical work allows us 
to find a restricted product interpretation of the image of the valuation. So it is convenient to provide some 
foundational discussion appropriate to this case.

Much more interesting is our adelic version of the Basarab-Kuhlmann structures \cite{basarab}\cite{kuhlmann} on local fields. 
We relate this to Krasner's hyperrings \cite{krasner} associated to local fields, 
and provide a natural quantifier-elimination for the adelic version.

Finally, we address the issues of stable embedding of the local fields in the adeles, 
stable interpretation of the value monoid of the adeles, and the property of not having the tree property of the second 
kind, $NTP_2$. 

All readers of \cite{FV} know the importance of enrichments of atomic Boolean algebras. A specially important case is 
$(Powerset(I), Fin)$, where $Powerset(I)$ is the powerset of $I$ and $Fin$ is the ideal of finite sets. In \cite{DM-boole} we showed 
that there is good elimination theory for various refinements, e.g. by $Even$, where $Even$ picks out the finite sets of even 
cardinality, or by predicates expressing congruence conditions on cardinality of finite sets. 
We hope that these refinements will find applications. 

We are able to work internally, in the language of ring theory, because $\A_K$ has lots of idempotents. It is not a 
von Neumann regular ring, so we are appealing to more than is used in the observations used by Kochen \cite{kochen} and Serre 
\cite[pp.389]{serre} that an ultraproduct of fields is canonically isomorphic to the product of the fields modulo a maximal ideal.

\section{Generalities}\label{sec-general}

The data for Theorems of Feferman-Vaught type consists of: 

(i) A (possibly many-sorted) first-order language $L$, which has the equality symbol $=$ of various sorts, 
and may have relation symbols and 
function-symbols. Convenient references for many-sorted logic are \cite{KK,feferman,pillay-book,MR}. One convention from 
\cite{pillay-book} that we choose not to follow is that the sorts be disjoint. This is an unnecessary restriction, especially 
when the only well-formed equality statements in our formalism demand that the terms involved be of the same sort.

(ii) $\mathcal{L}_0$, the usual language for Boolean algebra, with $\{0,1,\wedge,\vee,\bar{} \ \}$,

(iii) $\mathcal{L}$, {\it any} extension of $\mathcal{L}_0$,

(iv) $I$, an index set, with associated atomic Boolean algebra $Powerset(I)$ (the powerset of $I$, which will be denoted 
by $\B$ say),

$\B_{\mathcal{L}}$ will be {\it some} $\mathcal{L}$-structure on $\B$ where $\{0,1,\wedge,\vee,\bar{} \ \}$ 
have their usual interpretations.

(v) A family $\{M_i: i\in I\}$ of $L$-structures with product $\Pi=\prod_{i\in I} M_i$. 

One first forms, for each sort $\sigma$, the product 
$$\prod_{i\in I} Sort_{\sigma}(M_i),$$
where $Sort_{\sigma}(M_i)$, qua set, is just the $\sigma$-sort of $M_i$. 
This product is the $\sigma$-sort of the product $\Pi$. The elements are just functions 
$f_{\sigma}$ on $I$ with
$$f_{\sigma}(i) \in Sort_{\sigma}(M_i)$$
for all $i$. 

We generally write 
$$\bar f_{\sigma_1},\dots,\bar f_{\sigma_j},\dots$$
for tuples of elements of sorts 
%$\sigma$, and 
$\sigma_1,\dots,\sigma_j,\dots$ respectively; and 
%$$\bar w_{\sigma}$ for 
$$\bar x_{\sigma_1},\dots,\bar x_{\sigma_j},\dots$$
for tuples of $L$-variables of sorts 
%$$$\sigma$. 
$\sigma_1,\dots,\sigma_j,\dots$ respectively.

Suppose $\tau$ is a function-symbol of sort 
$$\sigma_1\times \dots \times \sigma_r \rightarrow \sigma.$$ 
Then the interpretation of $\tau$ in $\Pi$ is given by 
$$\tau^{(\Pi)}(\bar f_{\sigma_1},\dots,\bar f_{\sigma_r})(i)=
\tau^{(M_i)}(\bar f_{\sigma_1}(i),\dots,\bar f_{\sigma_r}(i)).$$
For an $L$-formula $\Phi(\bar w_{\sigma_1},\dots,\bar w_{\sigma_r})$ we define
$$[[\Phi(\bar f_{\sigma_1},\dots,\bar f_{\sigma_r})]]=\{i: M_i \models \Phi(\bar f_{\sigma_1}(i),\dots,\bar f_{\sigma_r}(i))\}.$$

The interpretation of a basic relation symbol $R$ of sort 
$$\sigma_1\times \dots \times \sigma_r$$
is given by
$$R^{(\Pi)}(\bar f_{\sigma_1},\dots,\bar f_{\sigma_r}) \Leftrightarrow 
[[R(\bar f_{\sigma_1},\dots,\bar f_{\sigma_r})]]=1.$$
In this way we have defined the natural product $L$-structure on $\Pi$, agreeing with the usual 1-sorted version.

We usually write $z_1,\dots,z_j,\dots$ for variables of the language $\cL$.

Now we bring in $\B_{\mathcal{L}}$, by defining new relations 
$$\Psi \circ < \Phi_1,\dots,\Phi_m>,$$
where $\Psi(z_1,\dots,z_m)$ is an $\mathcal{L}$-formula, and $\Phi_1,\dots,\Phi_m$ are $L$-formulas in a common set of 
variables $\bar x_{\sigma_1},\dots,\bar x_{\sigma_s}$ of sorts $\sigma_1,\dots,\sigma_s$ respectively, by:
$$\Pi\models \Psi \circ<\Phi_1,\dots, \Phi_m>(\bar f_{\sigma_1},\dots,\bar f_{\sigma_s}) \Leftrightarrow$$ 
$$\B_{\mathcal{L}}\models \Psi([[\Phi_1(\bar f_{\sigma_1},\dots,\bar f_{\sigma_s}),\dots,
[[\Phi_m(\bar f_{\sigma_1},\dots,\bar f_{\sigma_s})]]),$$
for $\bar f_{\sigma_1},\dots,\bar f_{\sigma_s}\in \Pi$.

We extend $L$ by adding a new relation symbol, of appropriate arity, for each of the above. In this way we get 
$L(\B_{\mathcal{L}})$, and $\Pi$ has been given an $L(\B_{\mathcal{L}})$-structure.

The results of Feferman-Vaught \cite{FV} are proved for one-sorted languages with no function symbols, though it is 
pointed out that their basic theorems on generalized products readily adapt to the more general case. 
The following, adequate for our adelic purposes, is a special case of an even more general theorem in \cite{FV}.
\begin{thm}\label{thm-fv} Uniformly for all families $\{M_i: i\in I\}$ the product $\Pi$ has constructive quantifier elimination in 
 $L(\B_{\mathcal{L}})$.
\end{thm}
\begin{proof} The proof of Theorem 3.1 in \cite[pp.65]{FV} goes through.
 \end{proof}

\section{Restricted products I}\label{res-pr1}
This is a more delicate matter. To our knowledge, the construction has been studied only 
for 1-sorted situations. We quickly review the definitions {\it in that case}.

So $L$ is assumed 1-sorted. Moreover, we assume $\mathcal{L}$ has a 1-ary relation symbol $Fin$, interpreted in $Powerset(I)$ 
as the set of finite subsets of $I$. 

Let $\Phi(x)$ be a fixed $L$-formula in a single variable $v$, 
and $\{M_i:i\in I\}$ be $L$-structures subject only to the constraint that each 
$\Phi(M_i)$ is an $L$-substructure of $M_i$ (here $\Phi(M_i)$ denotes the set defined by $\Phi(x)$ in $M_i$). 

With the above assumptions we define the restricted product of the 
$M_i$, ($i\in I$), relative to the formula 
$\Phi(x)$, denoted $\Pi^{(\Phi)}$ (or $\prod_{i\in I}^{(\Phi)} M_i$), as the $L$-substructure of 
$\Pi=\prod_{i\in I} M_i$ consisting of the $f$ such that
$$Fin([[\neg \Phi(f)]])$$
holds. (Our preceding assumptions make it an $L$-{\it substructure} of $\Pi$).

\begin{ex} $L$ is the language of group theory, with primitives 
$\{\cdot,^{-1},1\}$, and $\Phi(x)$ is the formula $x=1$. 
The $M_i$ are arbitrary. The restricted product is the direct sum.
 \end{ex}

\begin{ex} $L$ is the language of fields with a valuation ring, with primitives $\{+,-,\cdot,0,1,V\}$, where 
where $V(x)$ is a unary predicate for the valuation ring. $\Phi(x)$ is the formula $V(x)$. 

If $\{M_i: i\in I\}$ is the family of completions of an algebraic number field $K$ with respect to the normalized 
valuations of $K$, the restricted product is the ring of adeles of $K$, denoted $\A_K$ (cf.\cite{Cassels}).

If $\{M_i: i\in I\}$ is the family of completions of $K$ with respect to the non-archimedean normalized 
valuations of $K$, then restricted product is the ring of finite adeles of $K$, denoted $\A_{K}^{fin}$.
\end{ex}

In fact $\Pi^{(\Phi)}$ is an $L(\B_{\mathcal{L}})$-definable $L(\B_{\mathcal{L}})$-substructure of $\Pi$ (remember, 
$Fin$ is in $\mathcal{L}$) since the basic $L(\B_{\mathcal{L}})$-formulas not in $\mathcal{L}$ are all relational). 

%For any $L(\B_{\mathcal{L}})$-formula $\psi$ one has
%$$\Pi^{(\Phi)}\models \psi(\bar f) \Longleftrightarrow \Pi\models \psi^{\Pi^{(\Phi)}}(\bar f)$$
%$$\Longleftrightarrow \Pi\models \Theta(\bar f) \Longleftrightarrow \Pi^{(\Phi)}\models \Theta(\bar f),$$
%for some quantifier-free $L(\B_{\mathcal{L	}})$-formula $\Theta$, where $\psi^{\Pi^{(\Phi)}}$ denotes the 
%relativization of $\psi$ to $\Pi^{(\Phi)}$ which is defined inductively by replacing a quantified subformula 
%$$\exists z \psi(z)$$
%in $\psi$ by the quantified subformula
%$$\exists z Fin[[\neg \Phi(z)]]\wedge \psi(z).$$

\begin{thm}\label{fv-thm} $\Pi^{(\Phi)}$ has quantifier-elimination as an $L(\B_{\mathcal{L}})$-structure uniformly in 
 $\{M_i:i\in I\}$ and $\Phi(x)$, and effectively.
\end{thm}
\begin{proof} For an $L(\B_{\mathcal{L}})$-formula $\Psi$, let $\Psi^{\Pi^{(\Phi)}}$ denote the 
relativization of $\Psi$ to $\Pi^{(\Phi)}$ which is defined inductively by replacing a quantified subformula 
$$\exists y (\Psi(...,y,...))$$
(where $y$ is a single variable) in $\Psi$ by the quantified formula
$$\exists y (Fin([[\neg \Phi(y)]])\wedge \Psi(...,y,...)).$$
Now, by Theorem \ref{thm-fv}, we have for all tuples $\bar f$ from $\Pi$,
$$\Pi\models \Psi^{\Pi^{(\Phi)}}(\bar f)\Leftrightarrow \Pi\models \Theta(\bar f)$$
for a quantifier-free formula $\Theta$ from $L(\B_{\mathcal{L}})$. Finally, since 
$\Pi^{(\Phi)}$ is an $L(\B_{\mathcal{L}})$-definable $L(\B_{\mathcal{L}})$-substructure of $\Pi$, we get
$$\Pi^{(\Phi)}\models \Psi(\bar f) \Leftrightarrow \Pi\models \Psi^{\Pi^{(\Phi)}}(\bar f)$$
$$\Leftrightarrow \Pi\models \Theta(\bar f) \Leftrightarrow \Pi^{(\Phi)}\models \Theta(\bar f).$$
\end{proof}
\section{Restricted products II} We have experimented with various notions of restricted product in the many-sorted case. 
The notion explained below is the only viable notion we found. In this section, we give a many-sorted version of the results in 
Section \ref{res-pr1}. 

Assume $L$ is many-sorted, perhaps with both function-symbols and relation-symbols. Let $M, N$ be $L$-structures. 
We put $N_{\sigma}=Sort_{\sigma}(N)$ 
for every sort $\sigma$.
\begin{definition*} An $L$-morphism $F:N\rightarrow M$ is a collection of maps 
$$F_{\sigma}: N_{\sigma} \rightarrow M_{\sigma},$$
where $\sigma$ ranges over the sorts, such that for any relation symbol $R$ of sort 
$$\sigma_1\times \dots \times \sigma_k$$
we have, 
$$N_{\sigma_1}\times \dots \times N_{\sigma_k} \models R(\bar f_1,\dots,\bar f_k)\Leftrightarrow$$ 
$$M_{\sigma_1}\times \dots \times M_{\sigma_k} \models R(F_{\sigma_1}(\bar f_1),\dots,F_{\sigma_k}(\bar f_k)),$$
and for any function symbol $G$ of sort
$$\sigma_1\times \dots \times \sigma_k \rightarrow \sigma$$
we have 
$$G(F_{\sigma_1}(\bar f_1),\dots,F_{\sigma_k}(\bar f_k))=F_{\sigma}(G(\bar f_1,\dots,\bar f_k)),$$
where $\bar f_1,\dots,\bar f_k$ denote tuples of elements of sorts $\sigma_1,\dots,\sigma_k$ respectively.
\end{definition*}

\begin{note} Our convention that we have the usual equality as a binary relation on each sort forces each 
$F_{\sigma}$ to be injective. There is certainly a case for relaxing this convention or replacing 
$\Leftrightarrow$ by $\rightarrow$, but there is no gain for our present purposes.
\end{note}

If each $N_{\sigma} \subseteq M_{\sigma}$, and the identity maps constitute an $L$-morphism, we say 
{\it $N$ is an $L$-substructure of $M$}.
%A {\it substructure} $N$ of $M$ involves firstly 
%a collection of substructures of 
%$Sort_{\sigma}(M)$, say $N_{\sigma}$ for every sort $\sigma$. We have 
%$$N_{\sigma}=Sort_{\sigma}(N)$$ 
%for every $\sigma$. Note that if $F$ is a function 
%symbol of sort $\sigma \rightarrow \tau$, then for every $a\in Sort_{\sigma}(N)$, we must have 
%$$F(a)\in Sort_{\tau}(N).$$

Suppose that for each sort $\sigma$ we have a formula $\Phi_{\sigma}(x_{\sigma})$ in a 
single free variable $x_{\sigma}$ of 
sort $\sigma$, and we make the assumption that for each $\sigma$ for all $i$ the sets 
$$S_{\sigma,i}=\{x\in Sort_{\sigma}(M_i): M_i\models \Phi_{\sigma}(x)\}$$
naturally constitute an $L$-substructure of $M_i$. Note that in particular, for any function symbol $F$ of sort 
$$\sigma \rightarrow \tau$$
and any $a\in S_{\sigma}(M_i)$ we have that 
$$F(a)\in S_{\tau}(M_i),$$
for all $i$.

Then we define $\Pi^{(\Phi_{\sigma})}$ (also denoted $\prod_{i\in I}^{(\Phi_{\sigma})} M_i$), 
the restricted product with respect to the formulas $\Phi_{\sigma}(x)$, as the 
$L(\B_{\mathcal{L}})$-substructure of $\Pi$ consisting in sort $\sigma$ of the
$$f_{\sigma}\in \prod_{i\in I} S_{\sigma}(M_i)$$ 
such that 
$$Fin([[\neg \Phi_{\sigma}(f_{\sigma})]])$$
holds.
 
Note that $\Pi^{(\Phi_{\sigma})}$ is $L$-sorted: given $\sigma$, a sort of $L$, 
the $\sigma$-sort of $\Pi^{(\Phi_{\sigma})}$ is the set of all $f_{\sigma}\in \prod_{i\in I} S_{\sigma}(M_i)$ such that 
$$Fin([[\neg \Phi_{\sigma}(f_{\sigma})]])$$
holds. 

If $F$ is a function symbol of sort 
$$\sigma \rightarrow \tau,$$
and $a$ is in the $\sigma$-sort of $\Pi^{(\Phi_{\sigma})}$, then since the sets $S_{\sigma}(i)$ are $L$-substructures of $M_i$ for all 
$i$, we deduce that 
$$Fin([[\neg \Phi_{\tau}(F(f_{\sigma}))]])$$
holds. Hence $F(a)$ lies in $Sort_{\tau}(\Pi^{(\Phi_{\sigma})})$. 
Thus $\Pi^{(\Phi_{\sigma})}$ is a substructure of $\Pi$. It is clearly $L(\B_{\mathcal{L}})$-definable.

Now it is clear that the quantifier-elimination 
and effectivity of Theorem \ref{fv-thm} goes through. It is convenient to refer to the general version as Theorem 
$2_{\mathrm{sort}}$. 

Below we will give some important examples of restricted products associated to adeles, namely, the lattice-ordered 
value monoid and various hyperring structures connected to Basarab's formalism \cite{basarab}, which in turn connects 
to the much earlier work of Krasner \cite{krasner}.
\section{Eliminating the Boolean superstructure and quantifier elimination for adele rings} 
Already in Kochen \cite{kochen} (and surely in von Neumann's work) one sees that products 
$\prod_{i\in I} M_i$ of fields $M_i$ are von Neumann regular rings, with idempotents corresponding 
to subsets $S$ of $I$ via the correspondence 
$$S \longrightarrow e_S,$$
where $e_S$ is the idempotent defined by
$$e_S(i)=1,~i\in S,$$
$$e_S(i)=0,~i\notin S,$$
and conversely, an idempotent $e$ corresponds to the set 
$$S(e)=\{i\in I: e(i)=1\}.$$

Given an element $a=(f_1,f_2,\dots)\in \prod_{i\in I} M_i$, let $e_a$ be the idempotent 
corresponding to the support
$$supp(a)=\{i\in I: a(i)\neq 0\}.$$
Thus $e_a(i)=0$ if $a(i)=0$ and $e_a(i)=1$ if $a(i)\neq 0$. 

Note that $a$ and $e_a$ generate the same ideal. 
Indeed, it is clear that $a=(e_a)a$, thus $a$ is in the ideal generated 
by $e_a$. Conversely, let $b\in \prod_{i\in I} M_i$ be 
defined by $b(i)=0$ when $a(i)=0$ and $b(i)=a(i)^{-1}$ when $a(i)\neq 0$. 
Then $e_{a}=ab$, showing that $e_a$ is in the ideal generated by $a$. This shows von Neumann regularity of 
the products $\prod_{i\in I} M_i$ of fields $M_i$.

Kochen \cite{kochen} observes that maximal ideals in $\prod_{i\in I} M_i$ correspond to 
ultrafilters on the Boolean algebra $Powerset(I)$. The point is that in some cases of products one can code the 
Boolean algebra $Powerset(I)$ purely algebraically in the product ring $\prod_{i\in I} M_i$. In this way one can 
sometimes reconstruct the external Boolean apparatus inside the product ring. Note that there are limits to this, for example,  
one cannot define $Fin$ internally in any infinite product $\prod_{i\in I} M_i$. But, and this is crucial for us, 
we can define $Fin$ internally in the adeles $\A_K$ and in the finite adeles $\A_K^{fin}$ as is shown in \cite{DM-ad}. We 
remark that $\A_K$ and $\A_K^{fin}$ are not von Neumann regular (cf.\ \cite{DM-ad}).

The main novelty of \cite{DM-ad} over \cite{weispfenning-hab} is the internalization of \cite{FV} for the case of the 
adeles. In this way we get better quantifier-elimination. Below we briefly review the first-order definitions of the 
Boolean algebra, Boolean value, and the ideal $Fin$ in the case of adeles from \cite{DM-ad}. 

Let $K$ be a number field and $V_K$ (resp.\  $V_{K}^f$) denote the set of 
normalized valuations (resp.\ normalized non-archimedean (discrete) valuations) 
of $K$. Note the correspondence between subsets of $V_K^f$ and 
idempotents in $\Bbb A_K^{fin}$ given by 
$$S \longrightarrow e_S,$$ where 
$e_S(v)=1$ if $v\in S\subseteq V_K^f$, and $e_S=0$ if $v\notin S$. Clearly $e_S\in \Bbb A_K^{fin}$. Conversely, 
given an idempotent $e\in \Bbb A_K^{fin}$, let $S=\{v: e(v)=1\}$. Then $e=e_{S}$. 

There is a similar correspondence between subsets of $V_K$ and idempotents in $\A_K$.

Denote by $\B_K^f$ the Boolean algebra of idempotents of $\Bbb A_K^{fin}$ with the operations
$$e\wedge f=ef.$$
$$e\vee f=1-(1-e)(1-f)=e+f-ef.$$
$$\bar e=1-e.$$
$\B_K^f$ is quantifier-free definable 
in $\Bbb A_K^{fin}$ in the language of rings. 
Note that minimal idempotents $e$ correspond to normalized valuations $v_e$ of $K$, and vice-versa, 
$v$ corresponds to $e_{\{v\}}$ above. One has
$$\Bbb A_K^{fin}/(1-e)\Bbb A_K^{fin}\cong e\Bbb A_K^{fin}\cong K_{v_e}.$$ 
Note that the first of these structures is a definable quotient of $\A_K^{fin}$ and the second is a definable 
subring of $\A_K^{fin}$ with $e$ as a unit.

Similarly, one can define the Boolean algebra of idempotents in $\A_K$ and similar assertions hold.

Given a formula $\Phi(x_1,\dots,x_n)$ of the ring language, define $Loc(\Phi)$ as the set of all 
$$(e,a_1,\dots,a_n)\in\Bbb A_K^{n+1}$$
such that 
$e$ is a minimal idempotent and 
$$e\Bbb A_K^{fin} \models \Phi(ea_1,\dots,ea_n).$$
Here, $e\Bbb A_K^{fin}$ is 
a subring of $\Bbb A_K^{fin}$ with $e$ as its unit, is definable with the parameter $e$, and
$$e\A_K^{fin}\models \Phi(ea_1,\dots,ea_n) \Leftrightarrow \A_K^{fin}\models \Phi(ea_1,\dots,ea_n).$$
Note that $Loc(\Phi)$ is a definable subset of $(\A_K^{fin})^{n+1}$. 

Let $a_1,\dots,a_n\in \Bbb A_K^{fin}$. Define the Boolean value  
$$[[\Phi(a_1,\dots,a_n)]]$$
as the supremum of all the minimal idempotents $e$ in $\Bbb B_K^f$ such that 
$$(e,a_1,\dots,a_n)\in Loc(\Phi).$$
This is a definition in the language of rings which is uniform for all number fields $K$. If $\Phi$ has a string 
of quantifiers $Q_1x_1\dots Q_kx_k$, then the the formula defining $[[\Phi(a_1,\dots,a_n)]]$ 
has the string of quantifiers $\forall z Q_1x_1\dots Q_kx_k$, where $z$ is a variable distinct from the 
$x_1,\dots,x_k$.

The functions $\Bbb (\A_K^{fin})^n\rightarrow \Bbb A_K^{fin}$ given by 
$$(a_1,\dots,a_n)\rightarrow [[\Phi(a_1,\dots,a_n)]]$$ (for each formula $\Phi$) 
are definable in the language of rings, uniformly for all $K$. 
The support of an element $a\in \Bbb A_K^{fin}$, denoted $supp(a)$, 
is defined as $[[x\neq 0]]$0. 

We remark that the concepts of Boolean value can be defined similarly for the adele ring $\A_K$ and similar 
assertions hold. 

We shall denote by $\F_{K}^f$ (resp.\ $\F_{K}$) the ideal of idempotents in $\A_K^{fin}$ (resp.\ $\A_K$) whose 
support is finite.

The sets $\F_{K}^f$ (resp.\ $\F_{K}$) are definable in 
$\A_K^{fin}$ (resp.\ $\A_{K}$) in the language of rings. 
The defining formulas are existential-universal-existential (Cf.\ \cite{DM-ad}).

These results in conjunction with Theorem $2_{\mathrm{sort}}$ yield 
the following quantifier elimination theorem for 
$\A_K^{fin}$ in suitable extensions of the language of rings. 
\begin{thm}\cite{DM-ad}\label{qe-ad} $K$ be a number field. 
Let $L$ be a one-sorted (resp.\ many-sorted) extension of the language of rings 
in which the non-archimedean completions $K_v$ have 
uniform quantifier elimination (resp.\ uniform quantifier elimination in a sort $\sigma$). 
Let $\cL$ be any extension of the language of Boolean algebras containing a unary predicate 
$Fin(x)$ for the ideal of finite sets and unary predicates $C_j(x)$, for all $j\geq 1$, 
stating the there are at least $j$ distinct atoms below $x$. 
Let $\Phi(\bar x)$ be an $L$-formula. Then there are $L$-formulas $\Psi_1(\bar x),\dots,\Psi_m(\bar x)$ 
which are quantifier-free (resp.\ quantifier-free in sort $\sigma$) 
and a quantifier-free $\cL$-formula $\Theta(x_1,\dots,x_m)$ 
such that $\Phi(\bar x)$ is equivalent, modulo $Th(\A_K^{fin})$, to 
$$\Theta([[\Psi_1(\bar x)]],\dots,[[\Psi_m(\bar x)]]).$$
\end{thm}
Note that by the definition of $Fin$ in the language of rings, this is 
a quantifier elimination in $L$. Now using the quantifier elimination in the theory of infinite atomic Boolean algebras 
in the Boolean language enriched by the predicates $Fin(x), C_j(x)$, for all $j\geq 1$, (cf.\ \cite{DM-boole}), we 
deduce the following.
\begin{cor}\cite{DM-ad} \label{adelic-qe} A definable set $X\subseteq (\A_K^{fin})^m$ is a finite 
Boolean combination of sets of the form
 \begin{itemize}
  \item $Fin([[\Psi(\bar x)]])$,
  \item $C_j([[\Xi(\bar x)]])$,
\end{itemize}
where $\Psi(\bar x)$ and $\Xi(\bar x)$ are quantifier-free $L$-formulas (resp.\ quantifier-free $L$-formulas in sort $\sigma$), 
where $L$ is as in the Theorem \ref{qe-ad}.
\end{cor}
Note that the condition $[[\Psi(\bar x)]]=1$ is equivalent to $\neg C_1([[\neg \Psi(\bar x)]])$. 

These results imply similar results for the ring of adeles $\A_K$ (cf.\ \cite{DM-ad}). 

We will give examples of languages which can be used in Theorem \ref{qe-ad} and Corollary \ref{adelic-qe} in the next section. 

\section{Sortings of valued fields}\label{sec-sorts}

In the fifty years of the history of model theory of valued fields, 
many formalisms and sortings has proved useful. Each of these will provide 
a formalism for rings of finite adeles $\A_K^{fin}$. We have, as of now, seen no need to make 
a systematic study of all the possibilities, but some have engaged our attention. Below are some of the standard 
ingredients (it is intended that each sort has $=$ as a primitive). Given a local field $K$, $v$ denotes the valuation, 
$\cO_K$ the valuation ring, $\cM$ the maximal ideal, $U$ the unit group of 
$\cO_K$, and $\Gamma$ the value group.

(1) The field sort $K$ with primitives $\{+,-,.,0,1\}$,

(2) The multiplicative group sort $K^*$ with primitives $\{.,^{-1},1\}$

(3) The valuation ring sort $\cO_K$ with primitives $\{+,-,.,0,1\}$,

(4) The residue field sort $k$ with primitives $\{+,-,.,0,1\}$,

(5) The extended residue field sort $k\cup \{\infty\}$ with primitives $\{+,-,.,0,1,\infty\}$,

(6) The value group sort $\Gamma$ with primitives $\{+,-,0,<\}$,

(7) The extended value group sort $\Gamma\cup \{\infty\}$ with primitives $\{+,-,0,<,\infty\}$,

(8) The maximal ideal sort $\cM$, with primitives $\{+,-,.,0\}$,

(9) The $3$-sorted structure consisting of the sorts 
(1), (6), and (4) together with the connecting valuation and residue maps.

(10) The Basarab-Kuhlmann sorts $(K,K^*/1+\cM^n,\cO_K/\cM^m)$, for all $n\geq 1$, with primitives $\{.,^{-1},1\}$ for 
the sort $K^*/1+\cM^n$ and $\{+,-,.,0,1\}$ for the sorts $\cO_K/\cM^m$, with the valuation maps and canonical projection maps 
from the field sort into the other sorts.

(11) The many-sorted language $(K,K/1+\cM^n)$, for all $n$, with primitives $\{.,1,0,\Sigma\}$ 
with the valuation map, $\Sigma$ (the hyperring of Krasner, cf. Section \ref{sec-kras}), 
and the connecting map $K \rightarrow K/1+\cM^n$.

(12) The formalisms (10) and (11) with $\cM^n$ replaced by $\cM_{n,K}$ defined in Section \ref{sec-bk}.

(13) Valuations between appropriate sorts.

(14) Residue maps to residue rings.

(15) The place map from (1) to (5),

(16) The cross-section from (6) to (1) or (2),

(17) Angular component maps from $K$ to residue fields. Addition of this to the formalism (9) gives the Denef-Pas 
language. 

Thus there is a large stock of sortings relevant to the adeles. We will concentrate here 
on those connected to the value group sorts and the 
Basarab sorts. 

There are natural extensions of the above languages in which the completions $K_v$, where 
$v\in V_K^f$, of a number field $K$ have uniform quantifier elimination. These languages can be used in 
Theorem \ref{qe-ad} and Corollary \ref{adelic-qe}. A one-sorted example is the extension of the language of rings by 
the Macintyre predicates $P_n(x)$, for all $n\geq 1$, stating that $x$ is an $n$-th power, and the solvability predicates 
$Sol_k(y_0,\dots,y_n)$, $k\geq 1$, stating that $v(y_i)\geq 0$, for $0\leq i\leq n$, and 
the reduction of the polynomial 
$$x^n+y_1x^{n-1}+\dots+y_n$$ 
modulo the maximal ideal $\cM$ has a root in the residue field. Belair \cite{belair} proved that the $p$-adic fields 
$\Q_p$, for all $p$, have uniform quantifier elimination in this language and his proof carries over to the case of all $K_v$. 
Many-sorted examples are the language of Basarab and Kuhlmann (cf.\ \cite{basarab},\cite{kuhlmann} and Section \ref{sec-bk} below) 
sated in (10) above and the language of Denef-Pas \cite{Pas} stated in (17) above. 
In these languages the $K_v$ have uniform quantifier elimination in the field sort relative to the other sorts. 
These languages have been used in motivic integration (cf.\ \cite{DL},\cite{CL},\cite{HK}).

\section{The value monoid: various options}\label{val-monoid}
In connection with the value group, two particular sortings 
stand out, closely connected. The objective is to work out the meaning of these two on the finite adeles $\A_K^{fin}$. We also consider a 
third, connected to the ideles.

\

{\bf First version:} 

\

We have two sorts, corresponding to $K$ and $\Gamma\cup \{\infty\}$, and $v$ connecting them. 
$K$ has usual ring structure, $\Gamma \cup \{\infty\}$ has primitives 
$\{<,+,0,1\}$, and $\Gamma$ is an 
ordered abelian group. For technical reasons we need to 
replace the total order $<$ by lattice operations $\wedge,\vee$ (which 
are respectively min and max).
 
The axioms regarding $\infty$ in the value group sort are:
$$\infty+\infty=\infty,$$
$$\forall g\in \Gamma(\infty+g=\infty=g+\infty=\infty-g),$$
and
$$\forall g\in \Gamma(g<\infty).$$
The most natural thing is to have a constant $\infty$ of the extended value group sort.

On the factors $K_v$, $v(0)=\infty$, and for $x\in K_v^*$, 
$v(x)$ is the standard normalized valuation of $x$. This justifies the laws above. Note that $\infty-\infty$ is not defined, 
since $0/0$ is not. We justify $\infty-g$ by the remark that $0.x^{-1}=0$, $x\neq 0$.

Thus $\Gamma\cup\{\infty\}$ is a commutative monoid under the operation $+$, as is $\Gamma$ (which is in fact a group, 
though $\Gamma\cup\infty\}$ is not). So, at the cost of changing the notion 
of {\it substructure} there is a case, which we accept, for taking the inverse $-$ away from the basic formalism of the 
$\Gamma\cup\{\infty\}$ sort. $\Gamma\cup\{\infty\}$ is an ordered monoid, and $\Gamma$ an ordered group.

Given a number field $K$ with completion $K_v$ at the normalized non-archimedean discrete valuation $v$, 
the product $\prod_{v\in V_K^f} K_v$ has a product valuation $\prod v$ to
$$\prod_{v\in V_K^f} (\Gamma_v\cup \{\infty\})$$ 
where $\Gamma_v$ is $\Z$ for all $v$, and the Feferman-Vaught theory gives us a 
decidable model theory for this. Note that the product 
$\prod_{v\in V_K^f} (\Gamma \cup \{\infty\})$ is a lattice-ordered monoid. Note too that the image of 
$\A_K^{fin}$ is the set of $g$ such that $g(v)\geq 0$ for all but finitely many $v$. 
We shall see a bit later how to mimic in the $\Gamma\cup\{\infty\}$ 
sorting what we did in the adelic setting, i.e. give an internal definition of the Boolean value $[[\Phi(\bar x)]]$. 
That will involve a switch from 
$<$ to the lattice operations $\wedge,\vee$.

Our goal is to represent the product valuation on $\A_K^{fin}$ in terms of a {\it restricted} many-sorted formalism. Because of the 
substructure constraints in the general definition of restricted product, we proceed as follows. To the one-sorted formalism for the 
ring of adeles we add just one more sort, the {\it value sort}, which has as primitives $\{+,\wedge,\vee,0,\infty\}$. 
For a valued field $K$ these get their standard interpretation for $+$ and $0$ and $\infty$, but $\wedge$ and $\vee$ are respectively 
$min$ and $max$ in the ordering. 

Note that the following axioms are true in this value sort in the case of valued fields:

i) Axioms for lattice-ordered commutative monoids with $0$ as neutral element,

ii) Axioms about the distinguished element $\infty$, namely
 $$\infty+\infty=\infty,$$
%$$\infty+g=g+\infty, \ g\neq \infty,$$
%and
$$\forall g (\infty \wedge g=g),$$
$$\forall g (\infty \vee g=\infty).$$
These axioms are preserved under products.
%In the
%The formalism for this sort has primitives $\{+,0,\wedge,\vee,\infty\}$.
Note in contrast that the axiom special to the valued sort of a value field case, namely, 
$$\forall x\forall y(x\wedge y=x) \vee (x\wedge y=y),$$
is {\it not} preserved under products.

We want to carry this sorting to the adeles. So now we consider valued fields as 2-sorted structures 
consisting of a sort for the 
valued field, a sort for the lattice-ordered monoid with $\infty$, and a connecting map $v$. The product of these 
2-sorted structures will have in its first sort a von Neumann regular ring (as product of the field sorts), 
and in its second sort a lattice-ordered commutative monoid with distinguished element $\infty$ 
satisfying the axioms we gave before and in addition the following version of the valuation axioms:
$$\forall f\forall g (v(f.g)=v(f)+v(g)),$$
$$\forall f \forall g (v(f+g)\geq v(f) \wedge v(g)).$$
We are mainly interested in this product valuation on the finite adeles $\A_K^{fin}$. 
As remarked above, the image of the finite adeles $\A_K^{fin}$ under the 
product valuation is contained in the set of $g$ in 
$\prod_{v\in V_K^f} (\Gamma_v \cup \{\infty\})$ such that
$$Fin([[\neg(g\wedge 0=0)]])$$
holds. In fact the set of such $g$ is 
{\it exactly} the image of $\A_K^{fin}$ under the product valuation. This is immediate by lifting such a $g$ back to 
any $f$ with $v(f(v))=g(v)$.

Let us note that the pair 
$$(v(x)\geq 0,y\wedge 0=0)$$ 
satisfies the assumption in Section 4 that allows us to define a 
restricted product. So we can now identify, inside the 2-sorted structure with $K$ and $\Gamma\cup \{\infty\}$, and 
connecting valuation $v$, 
a natural restricted product, namely that with respect to the formulas $v(x)\geq 0$ in the $K$-sort, and the formula 
$y\wedge 0=0$ in the $\Gamma\cup\{\infty\}$ sort. This we call {\it the structure of $\A_K^{fin}$ with totally defined 
product valuation}. By 
Theorem $2_{\mathrm{sort}}$, it has a Feferman-Vaught quantifier-elimination. 

In Section 9 below we will go further, 
eliminating the Boolean scaffolding in the value group sort, in terms of the formalism of that sort. 

\

{\bf Second version:} 

\

We have three sorts, corresponding to $K$, $K^*$, and $\Gamma$, and 
$$v: K^*\rightarrow \Gamma,~~i:K^*\rightarrow K.$$
Again we will use $\wedge,\vee$ on $\Gamma$. Now there is no need for 
$\infty$. We take $K^*$ with primitives $\{.,1\}$, but not with the operation of inverse $\{^{-1}\}$. 

Obviously there is essentially no difference between the first and second versions in terms of expressive power. 
We could if needed make this precise in terms of bi-interpretability.

In the product we have $\prod_v v$, $\prod_v K^*_v$, and $\prod_v \Gamma_v$ and, now we get a restricted product using 
$v(x)\geq 0$ in the $K$-sort, $v(x)\geq 0$ in the $K^*$-sort, and $y\wedge 0=0$ in the $\Gamma$-sort.

Notice that the formula in the second sort actually involves the connecting map between the second and third sorts.

Now the restricted product that emerges consists of the finite adeles $\A_K^{fin}$ with the submonoid of elements with no 
zero coordinate and the product valuation from this set to the restricted product of the $\Gamma_v$. Call this the 
{\it $\infty$-free} restricted product for this version.

\

{\bf Third version:} 

\

Note however that another interesting possibility emerges if we take 
the formula of the middle sort to be $v(x)=0$ and the 
formula of the last sort to be $g=0$. Then the restricted product that emerges consists of the finite adeles with the finite 
ideles as a subgroup together with a valuation from it onto the direct sum of the value groups $\Gamma_v$ (a group!). 
Call this the {\it idelic} restricted product for this version.

The difference between the $\infty$-free and idelic restricted products for this version are:

(i) The former has three sorts, namely $\A_K^{fin}$, the submonoid of elements with no zero coordinates, and the value 
monoid sort.

(ii) The latter has three sorts, namely $\A_K^{fin}$, the ideles, and the submonoid of the value monoid from (i) consisting 
of elements which are zero at all but finitely many coordinates (i.e. a direct sum). 

We show later that (i) defines $(Powerset(V_K^f),Fin)$, and we can ``remove the Boolean scaffolding''. 

In (ii), it turns out that all $[[\Phi(f_1,\dots,f_n)]]$, where $f_j$ belong to the value monoid, 
are finite or cofinite, i.e. belong to the finite/cofinite
subalgebra of $Powerset(V_K^f)$. It turns out that we can {\it define} $Fin$, and then {\it interpret} the finite/cofinite algebra, 
and thereby ``remove the Boolean scaffolding''.

%There is an issue as to how one regards $\A_K$, in this formalism, as a restricted product. Note that the image of $\A_K$ in 
%the monoid sort is 
%the set of $f$ so that $[[f\geq 0]]$ is cofinite. The most significant point is that the natural interpretation of $f\geq 0$ is 
%not closed under $-$, and so to get a restricted product we had better (and do) remove $-$ from the ordered group sort. 
%If we do that then in the 
%field sort and the extended value group sort we have a restricted product, using $x\in V$ in the field sort and $f\geq 0$ in 
%the group sort, and 
%having $\prod v$ as connecting ``valuation''. Note that in the restricted product we have a lattice-ordered monoid rather 
%than a lattice-ordered 
%group.

\section{Interpreting the sorts in the field sort}

In \cite{CDLM} it is shown that the valuation ring is uniformly definable in all $K_v$ (by an existential-universal formula of 
the language of rings). From this it follows directly that all the 
sorts in Section \ref{sec-sorts} and the maps listed with them, together with the connecting maps between the sorts 
are uniformly interpretable in the field sort.

The angular component maps are known not to be interpretable, 
but have proved very useful, e.g.\ in motivic matters via the Pas language 
\cite{DL}. 

The ``corpoid'' or ``hyperring'' structure in (11) merits special attention. Fix $n$, and consider the group $K^*/1+\cM^n$ 
under multiplication. This is 
certainly interpretable. There is also a valuation $v$ on $K^*/1+\cM^n$ to $\Gamma$, 
the value group sort, clearly interpretable. Of course, the quotient 
$$\pi_n: K^*\rightarrow K^*/1+\cM^n$$
is interpretable. Finally, the relation $\Sigma_n$ which is the image of the 
graph of addition intersected with $(K^*)^3$ is interpretable and gives 
an ``approximation to addition``. 

Basarab \cite{basarab} showed that one has quantifier elimination for the field sort in terms of essentially extra sorts involving 
higher residue ring sorts  $\cO_K/\cM^n$ and the group sorts $K^*/1+\cM^n$. 

\begin{note} The Basarab construction works for general initial segments $I$ of the value group, 
but there is now no functorial sort. 
$I$ may not be interpretable.\end{note}

%Thus, the image of the ideles under the product of the $K^*/1+\mu_v^n$ (with $n$ fixed) is a semigroup version 
%of the value group situation and a restricted product.

\section{Removing the Boolean scaffolding in the value monoids of the $\A_K^{fin}$, for the totally defined, $\infty$-free, 
and idelic restricted products}\label{val-monoid2}

Recall that in the first and second versions discussed in Section \ref{val-monoid} we dealt respectively with

i) A restricted product involving two sorts, the usual valued field sort, and a value group sort which was a lattice-ordered 
monoid $\prod_{v\in V_K^f} (\Gamma_v\cup \{\infty\})$,

ii) A restricted product involving three sorts, the valued field sort, the multiplicative group sort, and a lattice-ordered 
monoid sort $\prod_{v\in V_K^f} \Gamma_v$.

There are only minor differences between these versions, whereas the third version is somewhat different.

\

{\bf Versions 1 and 2:} 

\

The restricted product is relative to the formula $v(x)\geq 0$ in the field sort, and the 
formula $y\wedge 0=0$ in the lattice-ordered monoid sort for both of the restricted products from 
(i) and (ii).
The restricted product is then the adeles with the (surjective) product valuation to the lattice-ordered monoid 
$$\{g: Fin(g\vee 0\neq g)\}$$ 
for both versions.

Now note that for each version the lattice-ordered monoid is itself a restricted product with respect to the formula 
$g\vee 0\neq g$ over the index set $V_K^f$. So the 
question arises as to whether we can eliminate the Boolean scaffolding for the restricted products. 
We do not have the machinery of idempotents 
which we exploited in $\A_K^{fin}$, so the problem is nontrivial. The following argument works for both versions as there is no reference 
to $\infty$.

How to interpret the elements of $V_K^f$? An {\it atom} of the lattice order is a minimal non-zero $e>0$. Such $e$ 
correspond exactly to the $g$ in the restricted product so that $g(v)=0$ except for a single $v_0$, where $g(v_0)=1$. 
The Boolean algebra $\B=Powerset(V_K^f)$ can be identified with the set of all $e$ which are either $0$ or a supremum of atoms. There is 
a largest such element which we call $1$. The Boolean operations on $\B$ are actually the lattice operations $\wedge,\vee$ of the 
lattice-ordered monoid. Note that the complexity 
of definition is higher than in the adele case.

How to define the finite elements of $\B$? Just note that $b$ in $\B$ is {\it finite} if and only if $b$ is invertible 
in the restricted product monoid. (We write $-b$ for the inverse). Note, of course, that we are now living a bit
dangerously notation-wise: $-b$ is the group-theoretic inverse (defined as the $c$ with $b+c=0$), and has little to do with the 
$-b$ in the 
Boolean ring. So again we see that the complexity of our definitions is greater than in the ring case. 
Thus we can interpret $(\B,Fin)$. It remains to define $[[\Phi(\bar x)]]$ 
and it suffices to define or interpret the stalk at an atom $e$, 
uniformly in $e$. 

One should note how the monoid operation $+$ relates to $\B$. The operation 
$+$ is not a group operation on the restricted product, but a trace of the operation 
$-$ survives on $\B$. Namely, if $e\in \B$, the Boolean complement $f$ of $e$ in $\B$ is $1-e$, i.e. $e+f=1$. 

Let $e$ be an atom, corresponding to a valuation $v$. The stalk at $e$ is just the lattice-ordered monoid $\Gamma_v$ 
(=$\Z\cup\{\infty\}$). We {\it identify} it with the substructure consisting of the $h$ such that $h(w)=0$ for all $w\neq v$. 
This we call the {\it internal stalk  at $e$}, and denote it by $\hat{\Gamma}_e$.
%\wedge annihilator of $(1-e)$
First suppose $h$ is any element of the restricted product 
with $h\wedge e=e$ (i.e. $h\geq e$). Then $h\geq 0$ and $h(v)\geq 1$ (the stalks are discretely ordered). If for some $w$, 
$h(w)\neq 0$, then $h(w)\geq 1$. 

Suppose $w_0\neq v$ and $h(w_0)\geq 1$. Define $j_1$ and $j_2$ by 
$$j_1(v)=2h(v),$$
$$j_1(w_0)=h(w_0),$$
$$j_2(v)=h(v),$$
$$j_2(w_0)=2h(w_0),$$
and if $w\neq v,w_0$
$$j_1(w)=h(w),\ j_2(w)=h(w).$$
Now $e\leq j_1\leq 2h$ and $e\leq j_2\leq 2h$, but {\it neither} $j_1\leq j_2$ {\it nor} $j_2\leq j_1$. Thus the 
interval $[e,2h]$ is not linearly ordered.

Conversely, if $h(w)=0$ for all $w\neq v$, then $[e,2h]$ is linearly ordered.

Next, suppose we have $h$ with 
$$h\wedge(-e)=h.$$ 
Then $h(w)\leq 0$ for $w\neq v$, and $h(v)\leq -1$. Then the preceding 
argument, mutatis mutandis, shows that $h(w)=0$ for all $w\neq v$ if and only if $[2h,-e]$ is linearly ordered. 

We conclude:

\begin{lem}\label{def-stalk} $h$ is in the stalk at $e$ if and only if either $h=0$ or $h\geq e$ and $[e,2h]$ is linearly ordered, or 
 $h\leq -e$ and $[2h,-e]$ is linearly ordered.
\end{lem}
\begin{proof} Done.
\end{proof}

This is however, not quite enough to get a definition of $[[\Phi(\bar x)]]$ in the style of what we did for $\A_K^{fin}$. 
We need to define the natural map from the value monoid to the stalk at $e$. Our restricted product is a structure of functions on $I$ 
(identified with set of atoms) and the stalk at $v\in I$ (which also call the {\it external stalk}) 
is the set of all $f(v)$, for $f$ in the restricted product. We now show how 
to interpret this.
For this, we show the following. Let 
$\Gamma_e$ (or $\Gamma_e\cup \{\infty\})$ 
denote the stalk at the atom $e$ defined as the set of $h$ in the product such that $h(v)=0$ for all atoms 
$v\neq e$.

By Lemma \ref{def-stalk} there is a definition, in the 
restricted product, for the internal stalk at $e$, $\hat{\Gamma}_e$, where $e$ is an atom. 
Define the 
relation $\equiv_e$ on the restricted product by
$$f\equiv_e g \Leftrightarrow f(e)=g(e).$$
This is a congruence for $\wedge,\vee,+,0,1$. If $f,g\geq 0$, then it is clear that
\begin{equation}\label{positive-stalk}
f(e)=g(e)\Leftrightarrow \forall h \in \hat{\Gamma}_e~(h\geq f\Leftrightarrow h\geq g),
\end{equation}
(i.e. $h\wedge f=f \Leftrightarrow h\wedge g=g$ holds in the restricted product).

So we can define
$$f\equiv_e g \Leftrightarrow \exists f^+,f^-,g^+,g^-~(f^+\geq 0\wedge g^+\geq 0\wedge f^-\leq 0\wedge g^-\leq 0$$
$$\wedge f=f^{+}+f^-\wedge g=g^{+}+g^- \wedge f^+(e)=g^+(e) \wedge f^-(e)=g^-(e)$$
This follows from applying \ref{positive-stalk} to the $f^+$ and $g^+$ and to $f^-$ and $g^-$ with 
the order reversed.
%Now we construe the stalk as a quotient 
%modulo a definable congruence. 
%We define:
%$$g\equiv h (\mathrm{mod}~e)$$
%if and only if
%$$\exists j (g=h+j \wedge j(w)=0)$$
%for all $w\neq v$, where $v$ is associated to $e$. Note that we have defined above the second conjunction.
%It is clear that the above defines a congruence for the lattice-ordered monoid structure (with $0,\infty$) 
%on the restricted product. 
Thus:
\begin{lem} The stalk $\Gamma_e$ at the atom $e$ is interpretable uniformly in $e$ in the restricted product.
\end{lem}
\begin{proof} Done.
\end{proof}
So we identify the stalk at $e$ with the set of congruence classes modulo $\equiv_e$, thereby 
giving a definable meaning to the condition that the 
stalk at $e$ satisfies
$$\Phi(f_1(e),\dots,f_k(e)),$$
and so we define 
$[[\Phi(\bar f)]]$ as the set of $e$'s where this holds, and have completed the removal of the Boolean scaffolding. 

%We turn to:

\

{\bf Version 3:} 

\

Now the valuation is defined on the group of finite ideles, and the value monoid is the direct sum 
of the $\Gamma_v$, $v\in V_K^f$. We define Boolean operations as in previous versions, but in 
this case we do not get a Boolean algebra, just a lattice because all elements are {\it finite}, there is no top element, and 
no element is complemented (though we have relative complements). 

Note that some of the discussion of Case 1 goes 
through, namely that giving the interpretation of the stalk at $e$ and the natural projection to the stalk. Thus 
we can define 
$$e\in [[\Phi(\bar f)]].$$
%but $[[\Phi]](\bar f)$ is not actually an element of the restricted
Now we show that we can define $Fin$ in the restricted product. Given a formula $\Phi(\bar x)$, we can define 
$Fin([[\Phi(\bar f)]]$ by 
$$Fin([[\Phi(\bar f)]]\Leftrightarrow \exists f\forall e (e\in [[\Phi(\bar f)]] \rightarrow e\leq f).$$
%the atoms are all finite. 
Let $\B_{fin/cofin}$ denote the Boolean algebra of finite and cofinite subsets of $Powerset(V_K^f)$. 
%However, one gets an interpretation 
%of classical Boolean formulas in various $[[\Phi_j]]$ using quantification over $e$'s.
\begin{lem} For any formula $\Phi(\bar x)$, and $\bar f$ from the direct sum $\bigoplus_{v\in V_K^f} \Gamma_v$, 
the Boolean value $[[\Phi(\bar f)]]$ belongs to $\B_{fin/cofin}$.
\end{lem}
\begin{proof}
For almost all atoms $e$, we have $\bar f(e)=0$. Hence for almost all atoms $e$, the formula
$\Phi(\bar f(e))$ is $\Phi(0)$, hence is either true or false in $\Z$.
\end{proof}
So we have in effect defined $Fin$, but no Boolean algebra. However, we can {\it interpret} 
the finite-cofinite algebra $\B_{fin/cofin}$ in the direct sum of the $\Gamma_v$. 
\begin{lem} The Boolean algebra $\B_{fin/cofin}$ is interpretable in $\bigoplus_{v\in V_K^f} \Gamma _v$.
\end{lem}
\begin{proof} We interpret a Boolean algebra $\B$ in $\B_0:=\bigoplus_{v\in V_K^f} \Gamma _v$ 
as follows. Choose an element $\beta \in \B_0 \setminus \{0\}$ and 
let
$$\B_{\beta}=\B_0 \times \{0\} \cup \B_0 \times \{\beta\}.$$ 
usual on $\B_0$, and on $\B_0 \times \{\beta\}$ define
$$(x,\beta) \wedge (y,\beta) := (x\vee y,\beta),$$
$$(x,\beta) \vee (y,\beta) := (x\wedge y,\beta).$$
and
$$(x,0)\wedge (y,\beta)=(x\wedge \overline{y},0),$$
where $x\wedge \overline{y}$ is defined as the supremum of atoms $\gamma$ such that 
$\gamma \leq x$ and $\gamma \nleqq y$. Put 
$$\overline{(x,0)}=(x,\beta),$$
$$\overline{(x,\beta)}=(x,0),$$
and
$$(x,0) \vee (x,\beta)=\neg(\overline{(x,0)}\wedge \overline{(y,\beta)})=\overline{((x,\beta) \wedge (y,0))}.$$
Thus $\B_{\beta}$ is a Boolean algebra. Clearly, different choices of $\beta$ give isomorphic Boolean 
algebras. 

Given $(x,0)\in \B_0\times \{0\}$, define $Fin((x,0))\Leftrightarrow Fin(x)$, and given 
$(x,\beta)\in \B_0\times \{\beta\}$, define $Fin((x,\beta))\Leftrightarrow \neg Fin(x)$.
\end{proof}

In any case, we have shown that the Boolean scaffolding can be removed, up to interpretation, 
but probably not up to definition, 
and the lattice-ordered monoid is decidable and has a quantifier-elimination in all the cases.

\begin{remark} In the language of Boolean algebras $\B_{fin/cofin}$ is an elementary substructure of $Powerset(V_K^f)$, 
 but not in the Boolean language with a predicate $Fin$ for finite subsets.
\end{remark}
\begin{proof} This follows from the quantifier elimination theorem for infinite atomic Boolean algebras 
in the Boolean language enriched by unary predicates $C_j(x)$ stating the there are at least $j$ distinct atoms 
below $x$ (cf.\ \cite{DM-boole}) since the Boolean algebras $\B_{fin/cofin}$ and $Powerset(V_K^f)$ have the same atoms.
\end{proof}

%\section{The Basarab sorts}

\section{The Basarab sorts and hyperrings}\label{sec-kras}
%\subsection{the interpretation}
%{\bf Hyperrings}
\

This notion of hyperring was defined by Krasner \cite{krasner} and used by Connes-Consani \cite{CC}. We recall this notion. 
A set $H$ is called a 
{\it canonical hypergroup} (cf.\ \cite{CC}) if there is multivalued addition
$$+: H\rightarrow Powerset(H)$$  
(where the variables $x,y,z$ range over elements in $H$) satisfying the following axioms:

(1) $\forall x \forall y (x+y=y+x)$,

(2) $\forall x \forall y ((x+y)+z=x+(y+z))$,

(3) $\forall x (0+x=x+0=x)$,

(4) $\forall x \exists !y(0\in x+y)$ ($y$ is written as $-x$),

(5) $\forall x \forall y \forall z (x\in y+z \Rightarrow z\in x-y)$ (=$x+(-y)$).

The operation $+$ is called hyperaddition. 
The hyperring axioms require in addition that multiplication gives a monoid with multiplicative identity, and we have 
$$\forall r\forall s\forall t (r(s+t)=rs+rt),$$
$$\forall r\forall s\forall t (s+t)r=sr+tr,$$
$$0\neq 1.$$
%Here $rA$ and $Ar$ have standard meaning, and the axioms are obvious in Krasner's case with $1=1G$, 
%provided $0\neq 1$ in the ring $R$. 

A hyperfield $H$ is a hyperring such that it's nonzero elements form a group under multiplication. 
%$H_{\Delta}$ is clearly a hyperfield. (cf.\ \cite{krasner},\cite{CC}).

%{\bf Hyperrings associated to valued fields and the Basarab sorts}

Let $K$ denote a local field. In \cite{krasner}, Krasner defined a hyperring associated to $K$. This definition can be slightly 
generalized as follows. 

Let $\Delta$ be a subset of $\Gamma$ with $0\in \Gamma$, and closed downwards in the sense that $g\leq h$ and $h\in \Delta$ imply 
$g\in \Delta$. Such a $\Delta$ is called here {\it convex}. Note that if $-g\in \Delta$ then $\Delta+g$ is also convex. 
%For $g\in \Gamma$, $I+g$ is convex, and 
%$0\in I+g$ if and only if $-g\in I$. 
We denote $\cM_{\Delta}=\{x: v(x)>\Delta\}$. This is an ideal in $\cO_K$ (since $\infty>\Gamma$). Clearly 
$1+\cM_{\Delta}$ is a subgroup of $U$. 

Let $G_{\Delta}$ be the group $K^*/1+\cM_{\Delta}$ and $R_{\Delta}$ the ring $\cO/\cM_{\Delta}$. Let $H_{\Delta}$ be the monoid 
$K/1+\cM_{\Delta}$ (of orbits for the action of $1+\cM_{\Delta}$ on $K$).  Note that the valuation $v$ is $0$ on $1+\cM_{\Delta}$, 
and so induces 
''valuations`` $v$ from $G_{\Delta}$ to $\Gamma$, and $H_{\Delta}$ to $\Gamma\cup\{\infty\}$. 
Let 
$$P_{\Delta}=\{x\in H_{\Delta}: v(x)\geq 0\},$$ and
$$U_{\Delta}=\{x\in H_{\Delta}: v(x)=0\}.$$
Note that $0\in P_{\Delta}$.

The set $H_{\Delta}$ carries the structure of a hyperfield. More generally, by the construction of Krasner \cite{krasner} (cf.\ 
also \cite{CC}), given a commutative unital ring $R$ and a subgroup $G$ of its multiplicative group, 
the set of all orbits of $R$ under $G$, denoted by $R/G$ carries the structure of a hyperring defined as follows:
\begin{itemize}
 \item Hyperaddition: $xG+yG=\{(xG+yG)/G\}$ (a {\it a subset} of $R/G$),
\item Multiplication:: $xG.yG=(xy)G$.
\end{itemize}
In the above we use the standard notations $A+B=\{a+b: a\in A, b\in B\}$, called the sumset of $A$ and $B$; 
and $A/G=\{aG: a\in A\}$ for a subset $A\subseteq R$. 
We are using $+$ for the hyperaddition by slight abuse of language since we use the same notation for the sumset of the $G$-orbits, 
but it will hopefully 
be clear from the context. Hyperaddition is a multi-valued addition. 

The axioms for canonical hypergroup are all satisfied in Krasner's construction $R/G$, with 
$0=0G$ and $-(xG)=(-x)G$. For uniqueness in Axiom (4), note that if $0=a+b$, where 
$a\in xG$, $b\in yG$, then
$$b=-a.g$$
for some $g\in G$, so
$b\in (-x)G$, so $yG=(-x)G$.

The other axioms are verified in \cite{krasner}, with $1=1G$, 
provided $0\neq 1$ in the ring $R$. 

Another useful way to think of the hyperaddition (following Krasner \cite{krasner}) is as follows. 
Given $xG$ and $yG$, the sumset $xG+yG$ is a union of cosets and the hyper sum 
$xG+yG$ is the set of these cosets. So
$$xG+yG=\{zG \in R/G: zG\subseteq xG+yG\},$$
where, by slight abuse of language, 
the sum on the right hand side is sumset, and on the left hand side is hyperaddition.

Model-theoretically, it is more natural to replace the ''hyperoperation`` $+$ by $\Sigma$, the graph of that operation, 
namely
$$H\models \Sigma(x,y,z) \Leftrightarrow z\in x+y,$$
and we will often use this version. 

We define the {\it language of hyperrings} to be the language with a 
predicate for multiplication, and predicate for $\Sigma$, and constants for $0,1$. This is a natural language 
for hyperrings. 

We do not take the time to write out the hyperring axioms in terms of the primitives $\{.,1,0,\Sigma\}$. 
This is easily done, and will often be used. 

In the model theory of Henselian valued fields $K$, some important work of Basarab \cite{basarab} and Kuhlmann \cite{kuhlmann} 
is closely related 
to the construction above. We take $R=K$, and $G$ to be $1+\cM_{\Delta}$, where $\Delta$ is an initial segment of the value 
group, $0\in \Delta$, and $\cM_{\Delta}$ is the ideal of elements of the valuation ring consisting of the $x$ with $v(x)>\Delta$. 
(We make no further restriction on $\Delta$).

The hyperring
$$(K/1+\cM_{\Delta},.,1,0,\Sigma)$$
is called (by us) the Krasner-Basarab hyperring associated to $\Delta$, and denoted $Kras(\Delta)$. It has some extra structure 
coming from the
valuation on $K$. Note that $1+\cM_{\Delta}$ is a subgroup of the units of $\cO_K$, and the action of $1+\cM_{\Delta}$ preserves 
the valuation. Thus the valuation induces a map
$$v: K/1+\cM_{\Delta}\longrightarrow \Gamma\cup\{\infty\}$$
satisfying $v(xy)=v(x)+v(y)$ with usual conventions about $v(x)+\infty$ and $\infty+\infty$.

Inside $K/1+\cM_{\Delta}$ we consider $\cO_K/1+\cM_{\Delta}$, a hyperring by the same construction. 
One checks easily that $K/1+\cM_{\Delta}$ is a hyperring extension of $\cO_K/1+\cM_{\Delta}$, in the sense of 
\cite{CC}, and $K/1+\cM_{\Delta}$ is a hyperfield. We denote by
$$\pi_{\Delta}: K\rightarrow K/1+\cM_{\Delta}$$
the canonical projection map.

The surjection $\cO_K\rightarrow \cO_K/1+\cM_{\Delta}$ clearly respects division. Since
$$v(x)=v(y)$$
holds in $\cO_K$ 
if and only if $x$ and $y$ divide each other, we can define unambiguously $v(x(1+\cM_{\Delta}))$ as $v(x)$. Then 
the relation
$$v(x)\leq v(y)$$
on $\cO_K/1+\cM_{\Delta}$ is definable by $x|y$ (which denotes $x$ divides $y$). 
Also every non-zero element 
in $K/1+\cM_{\Delta}$ is of the form $ab^{-1}$, with $a,b\in \cO_K/1+\cM_{\Delta}$. We have to check how $v$ relates to the 
hyperaddition $+$. 
In fact it is easily checked that
$$\Sigma(x,y,t) \Rightarrow v(t)\geq \mathrm{min}\{v(x),v(y)\}.$$

In \cite{basarab} and \cite{kuhlmann}, 
Basarab and Kuhlmann work with $K^*/1+\cM_{\Delta}$, i.e.\ a multiplicative group. This is 
part of the hyperring (in fact hyperfield) $H_{\Delta}$ (namely the multiplicative group of its nonzero elements), 
and is quantifier-free definable in $H_{\Delta}$ (since we have a constant for $0$).

They also use the (higher residue) rings $\cO_K/\cM_{\Delta}$, and we show in Section \ref{def-res-ring} 
that this is actually interpretable in $K/1+\cM_{\Delta}$, using the primitives 
$\{., \Sigma$ and $P_{\Delta}\}$ for all valued fields 
(and without $P_{\Delta}$ for all Henselian valued fields with finite or pseudofinite residue field). The definitions are 
uniform across all the stated fields and all $\Delta$.

Note that on the sort $K^*/1+\cM_{\Delta}$, with $v:K^*/1+\cM_{\Delta} \rightarrow \Gamma$, the extra structure of hyperring on 
$K^*/1+\cM_{\Delta}$ is given by the 3-place relation which is the image of the graph of addition on $(K^*)^3$. Taking $K_v$ to be 
the family of completions $K_v$ of a number field $F$ under a non-archimedean absolute value $v$, we 
have the maps of products
$$\prod_{v\in V_F^f} K_v^*\rightarrow \prod_{v\in V_K^f} K_v^*/1+\cM_{\Delta}\rightarrow \prod_{v\in V_F^f} \Gamma,$$
giving rise to several restricted products of fields and hyperfields, 
where the value monoid in the restricted product is what we considered in Sections \ref{val-monoid} and \ref{val-monoid2}, 
which will be studied in Section \ref{restricted-hyper}.

\section{Uniform definition of valuation on the hyperrings}\label{def-val}
In this section, we will show that $P_{\Delta}$ is definable in $H_{\Delta}$ 
uniformly for all Henselian valued fields $K$ with finite or pseudofinite residue field, for any 
convex subset $\Delta$ of $\Gamma$ containing $0$. The definition is an 
adaptation to the hyperfield situation of the definition given in \cite{CDLM} of $\cO_K$ in $K$ uniformly for all $K$ 
satisfying the above conditions. We use the notation of \cite{CDLM}. 

Let $P_2^{AS}(x)$ be the formula $\exists y (x=y^2+y)$. Let $T^{+}(x)$ be the formula 
$$x\neq 0 \wedge \neg P_2^{AS}(x) \wedge \neg P_2^{AS}(x^{-1}).$$
Let $P_2^{AS,Kras}(x)$ be the ''hyperversion`` of $P_2^{AS}(x)$, namely,
$$\exists y \ \Sigma(y^2,y,x).$$
Let $T^{+,Kras}(x)$ be
$$x\neq 0\wedge \neg P_2^{AS,Kras}(x) \wedge \neg P_2^{AS,Kras}(x^{-1}).$$
We need to review the use of $T^+$ in giving a uniform definition of $\cO_K$ in $K$, for $K$ Henselian with $k$ finite or 
pseudofinite (an assumption we now make, certainly true in all nonarchimedean completions of number fields). 

We consider $T^+(K)$ and $T^+(k)$ the sets defined in $K$, resp. $k$, by the formula $T^+(x)$.
\begin{lem}\label{val-0}\begin{itemize}
\item $T^+(K)$ is a subset of the units $\cO_K^{\times}$,
\item If $v(\alpha)=0$ and $res(\alpha)\in T^+(k)$, then $\alpha\in T^+(K)$.
\end{itemize}
\end{lem}
\begin{proof} Follows from \cite[Lemmas 2 and 3]{CDLM}.\end{proof}
We note the for $k$ pseudofinite, $T^+(k)$ is infinite (cf.\ \cite{CDLM}).

Much deeper is the following.
\begin{thm}\label{val-large} There exists an integer $N>0$ such that if $k$ has cardinal at least $N$, then
 $$\cO_K=\{a+b+cd: a,b,c,d \in T^+(K)\}.$$
\end{thm}
This is used to obtain the following comprehensive result:
\begin{thm}\label{val-l} There exists an integer $l>0$ such that for all $K$ as above
 $$\cO_K=\{0,1\}+\{a+b+cd: a,b,c,d\in T^+(K)\}$$
$$\cup \{x: \exists y (T^+(y)\wedge T^+(x^l-1+y))\},$$
where $A+B$ denotes the sumset of two sets $A$ and $B$.
\end{thm}
\begin{proof} The result follows from the proof of \cite[Theorem 2]{CDLM}.
 
\end{proof}
Now we take this definition and find a ''hyperversion``.
\begin{lem}\label{lift} Let $x\in K$. Suppose $T^{+,Kras}(x(1+\cM))$ holds in $H_{\Delta}$. Then $x\in T^+(K)$.
 
\end{lem}
\begin{proof} Obviously $x\neq 0$. If $P_2^{AS}(x)$ holds in $K$, then for some $y$ in $K$
 $$x=y^2+y.$$
 But then, taking $w=y(1+\cM)$
 $$H_{\Delta}\models \Sigma(w^2,w,x(1+\cM))$$ contradicting 
 $$H_{\Delta}\models T^{+,Kras}(x(1+\cM)).$$
 So 
 $$K\models \neg P_2^{AS}(x).$$
 Similarly
 $$K\models \neg P_2^{AS}(x^{-1}).$$
 \end{proof}
 \begin{lem}\label{square} Let $K$ be a valued field with residue characteristic different from $2$. 
 Let $x\in K$ be an element of value $0$. 
 Then $x$ is a square in $K$ if and only $x(1+\cM_{\Delta})$ is a square in $K/1+\cM_{\Delta}$.
 \end{lem}
\begin{proof} We only have to show the right to left direction. 
Suppose that $x(1+\cM_{\Delta})$ is a square in $K/1+\cM_{\Delta}$. Then 
for some $y\in K$,
$$x(1+\cM_{\Delta})=y^2(1+\cM_{\Delta}).$$
Hence $x-y^2\in \cM_{\Delta}$. 
Let $f(y):=x-y^2$. Then $f'(y)=2y$. Note that $v(y)=0$ (since $v(x-y^2)>0$ and $v(x)=0$). Thus $$v(f'(y))=v(2y)=0.$$ 
Applying Hensel's Lemma we deduce that $x$ is a square in $K$.
\end{proof}

\begin{lem}\label{image} Suppose $x\in K$ and $x\in T^+(K)$. Then $T^{+,Kras}(x(1+\cM))$ holds in $H_{\Delta}$.
 
\end{lem}
\begin{proof} The argument is divided into two cases of whether the residue characteristic is $2$ or not.

\

{\bf Case 1}: $k$ has characteristic $2$.

\

By Lemma \ref{val-0}, $v(x)=0$. So $x(1+\cM_{\Delta})\neq 0$. Suppose
$$P_2^{AS,Kras}(x(1+\cM_{\Delta}))$$
holds in $H_{\Delta}$. Then for some $y$,
\begin{equation}\label{sigma1}
\Sigma(y^2(1+\cM_{\Delta}),y(1+\cM_{\Delta}),x(1+\cM_{\Delta}))
\end{equation}
holds in $H_{\Delta}$. Then $y\neq 0$, and for some $\sigma,\tau$ in $K$ with 
$$y^2(1+\cM_{\Delta})=\sigma(1+\cM_{\Delta}),$$
$$y(1+\cM_{\Delta})=\tau(1+\cM_{\Delta}),$$
we have
$$(\sigma+\tau)(1+\cM_{\Delta})=x(1+\cM_{\Delta}).$$
Note that if one of $\sigma,\tau$ has negative valuation, then so has $y$ and then 
$$v(\sigma)\neq v(\tau)$$ 
and $v(x)<0$, a contradiction. So each of 
$y,\sigma,$ and $\tau$ has non-negative valuation. But if one has positive valuation, then all have, 
so $$v(\sigma+\tau)>0$$
while $v(x)=0$. So we conclude that 
$$v(y)=v(\sigma)=v(\tau)=0.$$ 

From \ref{sigma1} we have that 
$$x\in x(1+\cM_{\Delta})\subseteq y^2(1+\cM_{\Delta})+y(1+\cM_{\Delta}),$$
hence
$$x=y^2+y^2\lambda+y+y\rho,$$
for elements $\lambda,\ \rho\in \cM_{\Delta}$. Thus
$$v(y^2+y-x)>\Delta.$$
Let $f(y):=y^2+y-x$. So $f(y)\in \cM$. But
$$v(f'(y))=v(2y+1),$$
and $2y\in \cM$, hence $2y+1\notin \cM$ and $v(f'(y))=0$. 
By Hensel's Lemma, we get $P_2^{AS}(x)$. But $x\in T^+(K)$, contradiction. 
So
$$\neg P_2^{AS,Kras}(x(1+\cM_{\Delta}))$$
holds in $H_{\Delta}$. Similarly
$$\neg P_2^{AS,Kras}(x^{-1}(1+\cM_{\Delta}))$$
holds in $H_{\Delta}$. So 
$$T^{+,Kras}(x(1+\cM_{\Delta}))$$
holds in $H_{\Delta}$. This completes the proof in Case 1.

\

{\bf Case 2}: $k$ has characteristic different from $2$.

\

In this case it is easy to see that the condition $P_2^{AS}(x)$ is equivalent to the condition 
$P_2(1+4x)$ in both $K$ and in $K/1+\cM_{\Delta}$. 

As in Case 1 we know that $v(x)=0$, $x(1+\cM)\neq 0$, and we assume that
$$P_2^{AS,Kras}(x(1+\cM))$$
holds in $H_{\Delta}$. Thus $K$ satisfies
$$P_2(1+4(x(1+\cM_{\Delta}))).$$
Applying Lemma \ref{square} we deduce that 
$P_2(1+4x)$ holds in $K$. Hence $P_2^{AS}(x)$ holds in $K$. The proof is now completed as in Case 1.
\end{proof}

Note that Lemmas \ref{lift} and \ref{image} show for $x\in K$ that 
$$K\models T^+(x) \Leftrightarrow H_{\Delta}\models T^{+,Kras}(x(1+\cM_{\Delta})).$$
To complete our work, it is convenient to introduce in the hyperrings the definable predicate 
$\Sigma_3(x,y,z,t)$, defined as
$$\exists w (\Sigma(x,y,w) \wedge \Sigma(w,z,t)).$$
Now fix $l$ as in Theorem \ref{val-l}. Define $\Theta_1(X)$ as 
$$\exists A,B,C,D [T^{+,Kras}(A)\wedge T^{+,Kras}(B)\wedge T^{+,Kras}(C)$$
$$\wedge T^{+,Kras}(D) \wedge \Sigma_3(A,B,CD,X)]$$
and $\Theta_2(X)$ as
$$\exists Y \exists W (T^{+,Kras}(Y) \wedge T^{+,Kras}(W)\wedge \Sigma_3(X^l,-1,Y,W)).$$
Now define $\Theta^{Kras}(X)$ as 
$$\Theta_1(X)\vee \Theta_2(X) \vee \exists S (\Theta_2(S) \wedge \Sigma_3(X,-1,S)).$$
Then we have:

\begin{thm}\label{unif-def-val} Uniformly for all Henselian 
valued fields $K$ with finite or pseudofinite residue field we have,
 $$X\in P_{\Delta} \Leftrightarrow H_{\Delta}\models \Theta^{Kras}(X).$$
\end{thm}
\begin{proof} Suppose first $X\in P_{\Delta}$, and let $X=x(1+\cM_{\Delta})$. Then 
$v(x)\geq 0$. So by Theorem \ref{val-l}
$$K\models \exists a,b,c,d \in T^+(K) (x=a+b+cd) \vee \exists y (T^+(y)\wedge T^+(x^l-1+y)) \vee$$
$$\exists z (x=z+1 \wedge \exists w (T^+(w) \wedge T^+(z^l-1+w)).$$
So,
$$H_{\Delta}\models \Theta^{Kras}(x(1+\cM_{\Delta})).$$
Conversely, suppose
$$H_{\Delta}\models \Theta^{Kras}(x(1+\cM_{\Delta}).$$ 
This condition is a disjunction of three clauses and we examine each separately.

\begin{claim}\label{clause-1} $H_{\Delta}\models \Theta_1(X) \Rightarrow X\in P_{\Delta}.$
\end{claim}
\begin{proof} Assume that 
$$H_{\Delta}\models \exists A,B,C,D [T^{+,Kras}(A)\wedge T^{+,Kras}(B)\wedge T^{+,Kras}(C)$$
$$\wedge T^{+,Kras}(D) \wedge \Sigma_3(A,B,CD,X)].$$
Choose
$$A=a(1+\cM_{\Delta}),$$ 
$$B=b(1+\cM_{\Delta}),$$ 
$$C=c(1+\cM_{\Delta}),$$
and
$$D=d(1+\cM_{\Delta}),$$
where $a,b,c,d\in K$, to witness the quantifiers. By Lemma \ref{lift}
$$K\models T^+(a)\wedge T^+(b) \wedge T^+(c) \wedge T^+(d).$$
The meaning of $\Sigma_3(A,B,CD,X)$ is that there are $\alpha,\beta,\lambda,\lambda,x,\mu$ in $K$ such that 
$$a(1+\cM_{\Delta})=\alpha(1+\cM_{\Delta}),$$
$$b(1+\cM_{\Delta})=\beta(1+\cM_{\Delta}),$$
$$(\alpha+\beta)(1+\cM_{\Delta})=\lambda(1+\cM_{\Delta})=\lambda'(1+\cM_{\Delta}),$$
$$(cd)(1+\cM_{\Delta})=\mu(1+\cM_{\Delta}),$$
$$(\lambda'+\mu)(1+\cM_{\Delta})=x(1+\cM_{\Delta}).$$
Since
$$v(a)=v(b)=v(c)=v(d)=v(\mu)=0,$$
(by Lemma \ref{val-0}), also
$$v(\alpha)=v(\beta)=v(cd)=v(\mu)=0,$$
and
$$v(\lambda)=v(\alpha+\beta)\geq 0,$$
and
$$v(\lambda')\geq 0,$$
so
$$v(x)=v(\lambda'+\mu)\geq 0.$$
So $X\in P_{\Delta}$. 
\end{proof}
\begin{claim}\label{clause-2} 
$H_{\Delta}\models \Theta_2(X) \Rightarrow X\in P_{\Delta}.$
\end{claim}
\begin{proof} Assume that  
$$H_{\Delta}\models (\exists Y) (\exists W) (T^{+,Kras}(Y)\wedge T^{+,Kras}(W)\wedge \Sigma_3(X^l,-1,Y,W)).$$
Choose $Y=y(1+\cM_{\Delta})$ and $W=w(1+\cM_{\Delta})$, (where $y,w\in K$), to witness the quantifiers. By Lemma \ref{lift}, 
$$K\models T^+(y)\wedge T^+(w).$$
The meaning of $\Sigma_3(X^l,-1,Y,W)$ is that there are 
$x',\theta,y',\rho,\rho'$ in $K$ such that
$$x'(1+\cM_{\Delta})=x^l(1+\cM_{\Delta}),$$
$$\theta(1+\cM_{\Delta})=(-1)(1+\cM_{\Delta}),$$
$$(x'+\theta)(1+\cM_{\Delta})=\rho(1+\cM_{\Delta})=\rho'(1+\cM_{\Delta}),$$
$$y'(1+\cM_{\Delta})=y(1+\cM_{\Delta}),$$
$$(\rho'+y')(1+\cM_{\Delta})=w(1+\cM_{\Delta}).$$
Thus (by Lemma \ref{val-0}) 
$$v(y)=v(y')=v(w)=0.$$
Obviously $v(\theta)=0$. So
$$v(\rho'+\theta)=0,$$
Hence $v(\rho')\geq 0$. Thus $v(\rho')\geq 0$. So
$$v(x'+\theta)\geq 0.$$
So $v(x')\geq 0$. Since
$$v(x')=lv(x),$$
we deduce that $v(x)\geq 0$.
Therefore $X\in P_{\Delta}$, completing the proof of the claim.
\end{proof}
To prove the theorem, suppose that
$$H_{\Delta}\models \Theta^{Kras}(X).$$
Then either $\Theta_1(X)$ or $\Theta_2(X)$ holds, in which case we deduce from Claims \ref{clause-1} and \ref{clause-2} that 
$X\in P_{\Delta}$; or there exists $S$ such that both $\Theta_2(S)$ and 
$\Sigma_3(X,-1,S)$ hold. Choose $s\in K$ with 
$$S=s(1+\cM_{\Delta}).$$
By Claim \ref{clause-2}, $v(s)\geq 0$.

Since $\Sigma_3(X,-1,S)$, there is $e,f\in K$ such that
$$e(1+\cM_{\Delta})=x(1+\cM_{\Delta}),$$
$$f(1+\cM_{\Delta})=(-1)(1+\cM_{\Delta}),$$
and 
$$(e+f)(1+\cM_{\Delta})=s(1+\cM_{\Delta}).$$
Thus $v(f)=0$ and thus $v(e)\geq 0$. But $v(e)=v(x)$, and we deduce $X\in P_{\Delta}$. This proves the Theorem.
 
\end{proof}

\section{Uniform interpretation of the higher residue rings}\label{def-res-ring}

In this section we show that the higher residue ring $R_{\Delta}$ is interpretable in the hyperring 
$H_{\Delta}$ uniformly for all valued fields and all $\Delta$ if we have a predicate for the valuation ring 
of $H_{\Delta}$. We deduce that for all valued fields with finite or pseudofinite residue field, 
$R_{\Delta}$ is uniformly interpretable in $H_{\Delta}$, uniformly in $\Delta$. 

We start with the following.
\begin{lem} There is a well-defined surjective set map $\Psi_{\Delta}: P_{\Delta}\rightarrow R_{\Delta}$ with
$$\Psi_{\Delta}(a(1+\cM_{\Delta}))=a+\cM_{\Delta},$$
%Note that $\Psi_{\Delta}(0)=\cM_{\Delta}$.
for every $a\in \cO_K$.
\end{lem}
\begin{proof}
If
$$a(1+\cM_{\Delta})=b(1+\cM_{\Delta}),$$
then
$$ab^{-1} \in 1+\cM_{\Delta},$$
so since $v(b)\geq 0$, we have 
$$a\in b+\cM_{\Delta}.$$
\end{proof}
\begin{note} $\Psi_{\Delta}$ sends $0\in P_{\Delta}$ (which is $0(1+\cM_{\Delta})$) to $\cM_{\Delta}$. 
\end{note}

We need to understand the fibers $\Psi_{\Delta}^{-1}(a+\cM_{\Delta})$, where $a\in \cO_K$.

\begin{lem}\label{case-1} Suppose $v(a)\in {\Delta}$ and $a\in \cO_K$. 
Then the fiber $\Psi_{\Delta}^{-1}(a+\cM_{\Delta})$ is naturally 
isomorphic to
$$\{a\} \times (1+\cM_{I-v(a)})/1+\cM_{{\Delta}}.$$ 
In particular, it has cardinal $1$ if and only if $I-v(a)=I$ which is true if $v(a)=0$, i.e. $a\in U$ and 
$a(1+\cM_{\Delta})\in U_{\Delta}$.
\end{lem}
\begin{proof}
Suppose
$$b(1+\cM_{\Delta})\in \Psi_{\Delta}^{-1}(a+\cM_{\Delta}).$$
Then
$$b-a\in \cM_{\Delta},$$
so $v(a)=v(b)$, so
$$b=\theta a,\ v(\theta)=0.$$
Also, $a(1-\theta) \in \cM_{\Delta}$, so 
$$v(1-\theta)+v(a)>{\Delta},$$
so 
$$v(1-\theta)>{\Delta}-v(a).$$ 
Note that $0\in {\Delta}-v(a)$. We have
$$\theta\in 1+\cM_{{\Delta}-v(a)}.$$ 
Now
$$b(1+\cM_{\Delta})=a(1+\cM_{\Delta})$$
if and only if
$$\theta\in 1+\cM_{\Delta}.$$ 
The proof is complete.
\end{proof}

\begin{lem}\label{case-2} Suppose $v(a) \notin {\Delta}$, i.e. $v(a)>{\Delta}$, and $a\in \cO_K$. Then
$$\Psi_{\Delta}^{-1}(a)=\Psi_{\Delta}^{-1}(0)=\{\gamma: \gamma>{\Delta}\} \times U/1+\cM_{\Delta},$$
which is infinite if ${\Delta}\neq \Gamma$.
\end{lem}

\begin{proof} We have
$$a+\cM_{\Delta}=0\in R_{\Delta}.$$
Now suppose that
$$\Psi_{\Delta}(b(1+\cM_{\Delta}))=0.$$ Then 
$v(b)>{\Delta}$. So 
$$\Psi_{\Delta}^{-1}(0)=\{g\in K/1+\cM_K: v(g)>{\Delta}\}.$$ 

Suppose $v(a), v(b)>{\Delta}$ and
$$ab^{-1} \in 1+\cM_{\Delta},$$
then $v(a)=v(b)$, as in Lemma \ref{case-1}, 
$$b=\theta a$$
with $v(\theta)=0$. Now 
$$\theta \in 1+\cM_{\Delta},$$
and we are done. 
\end{proof}
%So in this case $$\Psi_{\Delta}^{-1}(0)=\{\gamma: \gamma>{\Delta}\} \times U/1+\cM_{\Delta},$$ infinite if ${\Delta}\neq \Gamma$K
\begin{lem} Let $A$ be the graph of addition on $R_{\Delta}$. Then
$$P_{\Delta}^3\cap \Sigma=\Psi_{\Delta}^{-1}(A).$$
\end{lem}
\begin{proof}Consider elements $a+\cM_{\Delta},\ b+\cM_{\Delta}\in R_{\Delta}$. Then by Lemma \ref{case-1} we have
$$\Psi_{\Delta}^{-1}(a+\cM_{\Delta})=(a+\epsilon)(1+\cM_{\Delta}),$$
$$\Psi_{\Delta}^{-1}(b+\cM_{\Delta})=(b+\tau)(1+\cM_{\Delta}),$$
where
$$v(\epsilon)>\Delta-v(a),$$
and 
$$v(\tau)>\Delta-v(b).$$
Denoting by $+$ the hyperaddition in the Krasner construction, we have that 
$$(a+\epsilon)(1+\cM_{\Delta})+(b+\tau)(1+\cM_{\Delta})=$$
$$\{((a+\epsilon)(1+m_{\Delta})+(b+\tau)(1+m'_{\Delta})): 
m_{\Delta}\in \cM_{\Delta}, m'_{\Delta}\in \cM_{\Delta}\}.$$
Now it is immediate that
$$\Psi_{\Delta}(((a+\epsilon)(1+m_{\Delta})+(b+\tau)(1+m'_{\Delta}))(1+\cM_{\Delta}))=a+b+\cM_{\Delta},$$
which completes the proof of the lemma.
\end{proof}
%We put on $G_{\Delta}$ a predicate for $P_{\Delta}$ and a ternary predicate $\Sigma_{\Delta}$ defined as 
%$\Psi_{\Delta}^{-1}(\Sigma)$, where $\Sigma$ is the graph of addition on $R_{\Delta}$. 
\begin{thm}\label{def-res-ring} $R_{\Delta}$ is interpretable in $(H_{\Delta},.,0,1,P_{\Delta})$ 
(in any language where we have a 
predicate for $P_{\Delta}$) uniformly for all valued fields and all $\Delta$.
\end{thm}
\begin{proof} Define an equivalence relation $E$ on $H_{\Delta}$ by $E(g,h)$ if and only if 
$$\Psi_{\Delta}(g)=\Psi_{\Delta}(h),$$ 
where $g,h\in P_{\Delta}$,
with one extra class for $H_{\Delta} \setminus P_{\Delta}$. 

We first show that $E$ is definable in the group $H_{\Delta}$.
\begin{claim} $\Psi_{\Delta}^{-1}(0)$ is definable.
\end{claim}
\begin{proof} If $g\in\Psi_{\Delta}^{-1}(0)$, then
$$g=\hat{g}(1+\cM_{\Delta}),$$
where $\hat{g}\in \cM_{\Delta}$. We claim that
%$$1(1+\cM_K)\subseteq \hat{g}(1+\cM_{\Delta})+1(1+\cM_{\Delta}),$$
%since an element of the left hand side is of the form $1+\sigma$, where $\sigma\in \cM_K$, and 
%an element of the right hand side is of the form
%$$1+\rho+\hat{g}+\hat{g}\tau,$$
%where $\rho, \tau\in \cM_K$. This shows that
$$H_{\Delta}\models \Sigma(1,g,1).$$
Indeed, this holds if and only if there are 
$\alpha,\beta \in K$ with
$$\alpha(1+\cM_{\Delta})=1(1+\cM_{\Delta}),$$
$$\beta(1+\cM_{\Delta})=\hat{g}(1+\cM_{\Delta}),$$
and
$$(\alpha+\beta)(1+\cM_{\Delta})=1(1+\cM_{\Delta}).$$
To satisfy this we take $\alpha=1$, $\beta=\hat{g}$, and we are done since 
$$(\alpha+\beta)(1+\cM_{\Delta})=(1+\hat{g})(1+\cM_{\Delta})=1(1+\cM_{\Delta}).$$
Conversely, suppose that
$$H_{\Delta}\models \Sigma(1,g,1).$$
Then choosing $g=\hat{g}(1+\cM_K)$ where $\hat{g}\in K$, we have 
$$1(1+\cM_K)\subseteq \hat{g}(1+\cM_K)+1(1+\cM_K).$$
Thus for any $\rho\in \cM_K$ there are $\lambda,\tau \in \cM_K$ such that
$$1+\rho=1+\tau+\hat{g}+\hat{g}\lambda.$$
Choose such a $\rho$ and get such $\lambda$ and $\tau$.
We deduce that
$$\hat{g}(1-\lambda)=\rho-\tau\in \cM_K.$$
Hence, $\hat{g}\in \cM_K$, and 
$$\Psi_{\Delta}(g)=0.$$
\end{proof}
\begin{claim} For any $g,h\in H_{\Delta}$ we have
$$E(g,h) \Leftrightarrow H_{\Delta}\models \Sigma(g,\Psi_{\Delta}^{-1}(0),h).$$
\end{claim}
\begin{proof} Clear.
\end{proof}
Now we interpret $R_{\Delta}$ setwise as the equivalence classes for $E$, 
and we have also given an interpretation of $0$ and $1$.

Denote the $E$-class of an element $g$ by $g_E$. We define addition and multiplication on the classes by
$$g_E+h_E=j_E,$$
where $j$ is such that
$$H_{\Delta}\models \Sigma(g,h,j),$$
and 
$$g_E.h_E=j_E,$$
where
$$j=gh.$$
It is easy to see that addition is well-defined. 

To show that multiplication is well-defined consider
$$a+\cM_{\Delta}\in R_{\Delta},$$
and 
$$b+\cM_{\Delta} \in R_{\Delta}.$$ 
We may assume that $v(a)\in {\Delta}$ and 
$v(b)\in {\Delta}$ otherwise we get the zero element after 
multiplying. Consider arbitrary elements 
$$a\theta(1+\cM_{\Delta})\in \Psi^{-1}(a+\cM_{\Delta}),$$
and
$$b\psi(1+\cM_{\Delta})\in \Psi^{-1}(b+\cM_{\Delta}),$$ 
where $\theta$ and $\psi$ have value zero (cf. Lemma \ref{case-1}). Since
$$\Psi_{\Delta}(a\theta(1+\cM_{\Delta})b\psi(1+\cM_{\Delta}))=ab+\cM_{\Delta},$$
we need to show that
$$ab\theta\psi+\cM_{\Delta}=ab+\cM_{\Delta}.$$
By proof of Lemma \ref{case-1},
$$\theta=1+\epsilon,$$
where $v(\epsilon)>{\Delta}-v(a)$, and 
$$\psi=1+\delta,$$
where $v(\delta)>{\Delta}-v(b)$. So 
$$\theta\psi=1+\epsilon+\delta+\epsilon\delta$$
where $v(\epsilon\delta)>{\Delta}-(v(a)+v(b))$. Thus
$$ab(1-\theta\psi)\in \cM_{\Delta},$$
hence
$$1-\theta\psi\in \cM_{{\Delta}-(v(a)+v(b))},$$
as required.
\end{proof}

We have thus defined a ring structure on $H_{\Delta}/E$ isomorphic, under the map $\Psi_{\Delta}$, to $R_{\Delta}$.

\begin{thm}\label{def-res-ring2} The higher residue ring $R_{\Delta}$ is interpretable in the hyperring $H_{\Delta}$ uniformly for all 
Henselian valued fields with finite or pseudofinite residue field and uniformly in $\Delta$
\end{thm}
\begin{proof} Use Theorem \ref{unif-def-val} together with Theorem \ref{def-res-ring}.
\end{proof}

%\begin{remark}Since $G_K$ is stable, we can not define or interpret $P_\Delta$ in $G_{\Delta}$ alone.\end{remark}

%\subsection{Axioms for $\Sigma_{\Delta}$ (needs elaboration)}

%We drop the ${\Delta}$ from the subscripts from now on.

%We put $ker=\{x: \Sigma(1,x,1\}$ (equals $\Psi^{-1}(0)$). We interpreted elements of $\Gamma$ as classes modulo $U_{\Delta}$. 
%We have a constant for $1\in G$. We can define a multivalued sum by $x+y=\{z: \Sigma(x,y,z)\}$. 

%(1) $G$ is an abelian group.

%(2) $P$ is a submonoid.
 
%(3) $\forall g (g\in P \vee g^{-1}\in P)$.

%(4) $v(g)\leq v(h) \leftrightarrow gh^{-1}\in P$.

%(5) $\forall x\forall y\exists z \Sigma(x,y,z)$.

%(7) axioms for the addition (needs elaboration).

%(8) $ker$ is closed under multiplication by elements of $P$.

%(8) axioms for $E$: 

%(8-1) $\Sigma(g,j,h) \wedge \Sigma(1,j,1) \rightarrow \Sigma(h,j,g)$.

%(8-2) $\Sigma(1,j,1) \rightarrow \Sigma(g,j,g)$.

\section{A variant of the Basarab-Kuhlmann quantifier elimination}\label{sec-bk}
In the preceding discussion, for a general $K$, there may be many $\Delta$ to consider. There is, however, one 
family $\Delta_n$ 
of particular significance, defined for all $K$.

To define $\Delta_n$, let $p$ be the {\it characteristic exponent} of $K$, so $p$ is the characteristic 
of the residue field $k$ if $k$ has prime characteristic, and $1$ if $k$ has characteristic $0$. 
%If $K$ has characteristic $0$, \underline{which we assume 
%henceforward}, $v(p)\neq \infty$ if $p$ is the characteristic exponent of $k$. ??

Take $\Delta_{p,n}$ as $\{g: g\leq nv(p)\}$, where $n\geq 0$ and $p$ is the characteristic exponent of $K$. Clearly 
$\Delta_{p,n}$ is convex. Note that $\cM_{\Delta_{p,0}}$ is the maximal ideal of $\cO_K$.
If $K$ has residue characteristic zero, we put $\Delta_n=\Delta_{1,n}$.
\begin{lem} If $K$ has characteristic exponent $1$, $\Delta_n=\{g: g\leq 0\}=\Delta_0$, for all $n$.
\end{lem}
\begin{proof} Trivial.
 
\end{proof}

Now we define the {\it principal Krasner-Basarab sort} as the 2-sorted structure 
$$(K,H_{\Delta_0},\pi_0),$$
where the sort $K$ has the language of rings, the sort $H_{\Delta_0}$ has the 
structure of hyperrings, (cf.\ Section \ref{sec-kras}), and there is a symbol for the natural connecting map
$$\pi_0: K \rightarrow H_{\Delta_0}.$$
We denote this structure by $Kras^B_{\Delta_0}(K)$. 

If $K$ has residue characteristic $0$, this is the only sorting we need. However, if $K$ has residue characteristic $p>0$ 
we need to consider the other convex sets $\Delta_{p,n}$, for $n\geq 0, p\geq 1$. 
In this case, we define the {\it Krasner-Basarab $p$-sorting} as the structure
$$(K,H_{\Delta_{p,n}},\pi_{p,n})$$
with the language of rings for the field sort $K$, the language of hyperrings for the sorts 
$H_{\Delta_{p,n}}$, and symbols for the canonical maps
$$\pi_{n,p}: K\rightarrow H_{{\Delta}_{p,n}}.$$
We denote this structure by $Kras^B_{\Delta_{p,n}}(K)$.

Note that $\Delta_{p,m}\subseteq \Delta_{p,m+1}$, and we have natural (commuting) maps
$$K\longrightarrow H_{\Delta_{p,m+1}}\longrightarrow H_{\Delta_{p,m}}.$$
Note that $Kras^B_{{\Delta}_{1,n}}=Kras^B_{\Delta_0}$, (the principal sort).

We define the {\it Krasner-Basarab language} to be the many-sorted language consisting of 
the language of rings for the field sort, the language of hyperrings for the 
sorts $H_{{\Delta}_{p,n}}$, and function symbols for the connecting maps $\pi_{p,n}$ between the 
two sorts, for all $n\geq 0,p\geq 1$.

%Thus the Krasner-Basarab language is the union of the principal sort and $p$-sorting.
The theorem to be stated below, combining results of Basarab \cite{basarab} and Kuhlmann \cite{kuhlmann} 
is one of the most comprehensive and important in the 
model theory of Henselian fields. We do not present the most general version. 

Let $K$ be a valued field $K$. Given $n\geq 0$, we denote
$$G_K^n=K^*/1+\cM_{K,n},$$
and 
$$\cO_K^n=\cO_K/\cM_{K,n},$$
where 
$$\cM_{K,n}=\{a\in \cO_K: v(a)>nv(p)\}.$$
We denote the corresponding canonical projection maps by 
$$\pi_n:\cO_K\rightarrow \cO_K^n,$$ 
and 
$$\pi^*_n:K^*\rightarrow G_K^n.$$
Let $\Theta_n \subseteq \cO_K^n \times G_K^n$ be the relation defined by
$$\Theta_n(x,y) \Leftrightarrow \exists z\in \cO_K (\pi_n(z)=x \wedge \pi_n^*(z)=y).$$
%We denote by 
%$$v_n:\cO_{K,n} \cup G_{K,n}\rightarrow \Gamma\cup \{\infty\}$$ 
%the induced valuation. 
%Note that $H_{\Delta_{n,p}}=G_K^n$ and $\pi_{n,p}=\pi^*$ if $p$ is characteristic exponent of $K$.

Given a valued $K$, the {\it Basarab-Kuhlmann language} for $K$ is the many-sorted language with sorts:
$$\mathcal{K}_n:=(K, \cO_K^n, G_K^n, \pi_n, \pi^*_n, \Theta_n),$$
for all $n\geq 0$, 
with the language of rings for the field sort $K$ and the higher residue 
ring sorts $\cO_K^n$, and the language of groups for the sorts $G_K^n$. 
This language was defined by Kuhlmann \cite{kuhlmann} based on the language of Basarab \cite{basarab}. 

Given a structure $S$ with many sorts 
$$(\sigma_0(S),\sigma_1(S),\sigma_2(S),\dots),$$ where $\sigma_0(S)=S$ 
the home sort, 
and a formula $\Psi$ of the many-sorted language for $S$ (with free variables from the different sorts), we write
$$(S,\sigma_1,\sigma_2,\dots)\models \Psi$$
to indicate that the subformulas in $\Psi$ from the sort $\sigma_j$ hold in $\sigma_j(S)$, for all $j\geq 0$.

\begin{thm}\label{thm-b-k} Given a formula 
$\Phi(\bar x)$ of the Basarab-Kuhlmann language, there is an integer $\beta(\Phi(\bar x))$, a 
formula $\Psi(\bar x,\bar y)$ (in extra variables from the sorts other than the field), 
and integers $\gamma(p)$ for every prime $p\leq \beta(\Phi(\bar x))$, such that for every $\bar a$ from $K$,
$$K\models \Phi(\bar a) \Leftrightarrow \mathcal{K}_0 \models \Psi(\bar a,\pi_0(\bar a),\pi^*(\bar a)),$$
if the residue characteristic of $K$ is greater than $\beta(\Phi(\bar x))$, and 
$$K\models \Phi(\bar a) \Leftrightarrow \mathcal{K}_{\gamma(p)} \models \Psi(\bar a,\pi_{\gamma(p)}(\bar a),
\pi^*_{\gamma(p)}(\bar a)),$$
if the residue characteristic of $K$ is $p$, for every prime $p\leq \beta(\Phi(\bar x))$.
\end{thm}
\begin{proof} By Theorem 2.4 in \cite{kuhlmann} there is a formula $\Psi_1(\bar x,\bar y)$ 
involving the sorts $K$, $\cO_K^0$ and $G_K^0$ (with the 
appropriate maps) such that if $K$ has residue characteristic zero, then for every $\bar a$ from $K$ 
$$K\models \Phi(\bar a) \Leftrightarrow \mathcal{K}_0\models \Psi_1(\bar a,\pi(\bar a),\pi^*(\bar a)).$$
Using a compactness argument, this holds for $K$ of residue characteristic larger than some positive integer 
$\beta(\Phi(\bar x))$ depending only on $\Phi(\bar x)$. 
The case of $K$ with residue characteristic $p\leq \beta(\Phi(\bar x))$ follows from 
Theorem B in \cite{basarab}.
\end{proof}
We do not know if there is a uniform quantifier elimination for Henselian fields of 
residue characteristic $p$. Theorem \ref{thm-b-k} does not apply to this case because of the dependency of the 
ideals $\cM_{K,n}$ on $v(p)$. 

Theorem \ref{thm-b-k} holds in a very general setting. We can 
start with a notion of first-order formula of the language of valued fields which can be many-sorted (e.g. the 
standard 3-sorted (9) from Section \ref{sec-sorts}, 
with the field sort most prominent). The only restriction on the other sorts is that they be 
interpretable in the 3-sorted case, and that the value and residue sorts be interpretable in them. This restriction 
excludes angular components and cross-sections, but it is easy to prove a version of the theorem taking account of those. 

We can formulate the Basarab-Kuhlmann Theorem in the Krasner-Basarab language of hyperrings.

\begin{thm}\label{thm-kb} There is a computable map, defined on first-order formulas of the language of valued fields, 
assigning to each $\Phi(\bar x)$

i) an integer $\beta(\Phi(\bar x))$,

ii) a formula $\Phi_{0}(\bar x,\bar y)$ 
from the 2-sorted Krasner-Basarab language with field sort and principal sort, 
and having no bound variables of field sort, 

 iii) for each prime $p\leq \beta(\Phi(\bar x))$ 
an integer $\gamma_p(\Phi(\bar x))$ and formulas 
$$\Phi_{p,1}(\bar x,\bar y),\dots,\Phi_{p,r_p}(\bar x,\bar y)$$
from the $p$-sorting, with no bound variables of field sort, and no sorts $(p,n)$ for $n>\gamma_p(\Phi(\bar x))$, 

{\it such that} for all Henselian valued fields $K$ of characteristic $0$, 
if the residue characteristic of $K$ is not a prime 
at most $\beta(\Phi(\bar x))$ then for every $\bar a$ from $K$
$$K\models \Phi(\bar a) \Leftrightarrow Kras^B_{\Delta_0}\models \Phi_{0}(\bar a,\pi_0(\bar a)),$$
and if the residue characteristic of $K$ is a prime $p\leq \beta(\Phi(\bar x))$, then for some $r$ for all $\bar a$ 
from $K$ 
$$K\models \Phi(\bar a) \Leftrightarrow Kras^B_{\Delta{p,r}}\models \Phi_{p,r}(\bar a,\pi_r(\bar a)).$$
\end{thm}
\begin{proof} Combine Theorem \ref{thm-b-k} with the the interpretation of the higher residue rings in 
Theorem \ref{def-res-ring2}.
\end{proof}

\begin{note} Theorem \ref{thm-kb} remains true without the condition on the residue field provided we add a 
 predicate for the valuation ring of the hyperring to the (field sort of the) 
language of Krasner-Basarab.
\end{note}
 
%[{\bf previous version} 

%\begin{thm}\label{thm-kb} There is a computable map, defined on first-order formulas of the language of valued fields, 
%assigning to each $\Phi$

%i) an integer $\beta(\Phi)$

%ii) formulas $\Phi_{0,1},\dots,\Phi_{0,r_0}$ from the 2-sorted Krasner-Basarab language with field sort and principal sort, 
%and having no bound variables of field sort, 

%iii) for each prime $p\leq \beta(\Phi)$ an integer $\gamma_p(\Phi)$ and formulas 
%$$\Phi_{p,1},\dots,\Phi_{p,r_p}$$
%from the $p$-sorting, with no bound variables of field sort, and no sorts $(p,n)$ for $n>\gamma_p(\Phi)$ 
%{\it such that} for all Henselian valued fields $K$ of characteristic $0$, if the residue characteristic of $K$ is not a prime 
%at most $\beta(\Phi)$ then for some $r$
%$$K\models \Phi \Leftrightarrow \Phi_{0,r},$$
%and if the residue characteristic of $K$ is a prime $p\leq \beta(\Phi)$, then for some $r$
%$$K\models \Phi \Leftrightarrow \Phi_{p,r}.$$

%\end{thm}
%\begin{proof} Combine Theorem \ref{thm-b-k} with the the interpretation of the higher residue rings in 
%Theorem \ref{def-res-ring2}.
%\end{proof}]

\begin{note}\begin{itemize}\item If the prime $p$ is infinitely ramified, our understanding of the $H_{\Delta_{n,p}}$ 
is very limited. This stands in the way of proving decidability of the class of all finite extensions of $\Q_p$.
\item The preceding theorem does give much insight into definability in adele rings. It now leads us to 
look at adelic versions of 
the Krasner-Basarab hyperfields. \end{itemize}
\end{note}

\section{Restricted products of the hyperfields}\label{restricted-hyper}
%\subsection{}\label{13-1} 
In this section we consider adelic versions of the Krasner-Basarab structures. We need to slightly extend our 
notations.

As before, $K$ will be a number field, $V_K^f$ the set of normalized non-archimedean valuations of $K$, 
and $K_v$ the completion of $K$ at $v\in V_K^f$. We will let $S$ denote a finite subset of $V_K^f$.
%In this case, we denote
%$$H_{\Delta_{m,p(v),v}}=K_v/1+\cM_{K_v,m}$$
%and 
%$$\pi_{m,p(v),v}: K_v\rightarrow H_{\Delta_{m,p,v}}$$
%the canonical projection map, where $p(v)$ is the characteristic of the residue field of $K_v$. Note that in this case 
%$\Delta_{m,p(v),v}=\{g: g\geq mv(p)\}$, i.e., the same as the $\Delta_{m,p}$ defined earlier but for the field $K_v$. 
%In the case when $p=1$, we denote
%$$H_{\Delta_{0,v}}=H_{\Delta_{m,1,v}}$$
%and 
%$$\pi_{0,v}=\pi_{m,1,v}.$$
%We define the {\it language of hyperrings} to be the language with primitives $\{.,\Sigma,0,1\}$, where $\Sigma$ is the 
%graph of the 
%hyperaddition. This is a natural language for hyperrings. We consider the hyperrings 
%$H_{\Delta_{m,p(v),v}}$, for varying $v$ (and hence $p(v)$) and consider various restricted products arising from them.

Given $K$, we can consider several many-sorted restricted products constructed from the Krasner-Basarab structures 
on $Kras_{\Delta_{p,n}}(K)$ associated to $K$, for $p\geq 1, n\geq 0$. Given $K_v$, where $v\in V_K^f$, 
we will consider $Kras_{\Delta_{p(v),n}}(K_v)$, where $p(v)$ is the residue characteristic of $K_v$. We will denote 
$$\pi_{n,p(v)}:K\rightarrow Kras_{\Delta_{p(v),n}}(K_v)$$
the projection map.

The Krasner-Basarab language will give natural languages for these 
restricted products in the formalism of Section \ref{sec-general}. 
We will be concerned here mainly with the 2-sorted structure 
consisting of the the restricted product of the $K_v$ with the ring structure in the first sort, 
the restricted product of the $H_{\Delta_0,v}$ with the hyperring structure in the second sort, 
and the connecting map between the two sorts, for all $v\in V_K^f\setminus S$. 
This restricted product will be called {\it the $S$-adelic principal 
Krasner-Basarab structure associated to $K$}, and denoted 
$Kras^{\A}_{S,0}(K)$. Note that in particular, $S$ can be empty, in which case we have the {\it principal 
adelic Krasner-Basarab structure}. In this case, since the connecting map between the sorts is 
surjective, the finite adeles map onto the $f$ in $\prod_v H_{\Delta_0,v}$ with 
$[[\neg P_{\Delta}(f)]]$ finite, and the image of $\A_K^{fin}$ is 
the restricted product of the $H_{\Delta_0,v}$ with respect to $P_{\Delta}$. Note that 
$H_{\Delta_0,v}$ is relational, except for the monoid operation. 

Note that $P_{\Delta}$ uniformly defines the valuation on the hyperrings $H_{\Delta_{p(v),n}}$, for all 
$v\in V_K^f$, and all $n\geq 1$, 
by Theorem \ref{unif-def-val}. Thus adelic Krasner-Basarab structures can be construed as 
man-sorted restricted products in the formalism of Section \ref{sec-general}. One takes for the first sort 
a ring formula which uniformly defines the valuation rings of the $K_v$, and for the second sort a formula 
of the language of hyperrings which uniformly defines the valuation of the hyperrings associated to $K_v$, for all 
$v\in V_K^f$.

One can also include other sorts $H_{\Delta_{n,p}}$ for all the $v\in S$ 
(where $p$ is the residue characteristic of $K_v$). We call the resulting restricted product 
the {\it adelic-$(p,n)$ Krasner-Basarab structure}. There are certainly several other possibilities of 
restricted products constructed from the family of $K_v$ and $Kras_{\Delta_{p(v),n}}(K_v)$ 
for varying or fixed $n$ or $p$. 
Note that given $p,n$, $\Delta_{p,n}$ is uniformly definable for all local fields $K$ (and even in much more 
generality, cf.\ Section \ref{def-val}). For residue characteristic zero, $\Delta_{p,n}=\Delta_0$.

For both the $S$-principal and the $(p,n)$-adelic Krasner-Basarab structures, 
it is easy to interpret the Boolean algebra $\B$ with $Fin$ as follows. 
$\B$ is just the set of idempotents, with order $\leq$ defined using 
multiplication. The atoms are defined as usual. The stalk at an atom $e$ is naturally identified using the 
idempotents, (e.g.\ in the case of $S$-adelic principal Krasner-Basarab structures the stalk is 
$eKras^{\A}_{S,0}(K)$). This allows us to define the Boolean value 
$[[\Phi(\bar x)]]$, for a formula $\Phi(\bar x)$, as in the basic adelic situation. Finally $Fin$ is defined using the 
Boolean value $[[..]]$ and the valuation.

So for the adelic Krasner-Basarab structures 
we have a Feferman-Vaught Theorem, and we can eliminate the Boolean scaffolding with $Fin$.

\begin{thm} 
%The ring of finite adeles $\A_K^{fin}$ is bi-interpretable with the restricted product of the hyperfields 
%$\prod_{v\in V_K^f}^{P(v)} H_{\Delta_{0,v}}$ via the canonical maps $\pi_v$ from $K_v$ into $K_v/1+M$.
Let $K$ be a number field. Given a formula $\Phi(\bar x)$ of the language of rings (resp.\ the language of Basarab-Kuhlmann), 
there is an effectively computable finite set $S=\{v_1,\dots,v_s\}$ of normalized non-Archimedean absolute values of $K$ 
such that, for any $\bar a$ from $\A_K^{fin}$, the condition 
%$$\phi(x),\psi_1(\bar x),\dots,\psi_s(\bar x)$ of the language of hyperrings such that 
$$\A_K^{fin} \models \Phi(\bar a),$$
is equivalent to a Boolean combination of the following conditions:
\begin{itemize}
 \item (type I)
$$Kras_{S,0}^{\A}(K) \models Fin([[\Psi(\bar a,\pi_{0,v}(\bar a))]]),$$
%where $\psi(\bar x)$ is a formula of the 
%language of hyperrings, 
\item (type II) 
$$Kras_{S,0}^{\A}(K)\models C_k([[\Psi'(\bar a,\pi_{0,v}(\bar a))]]),$$
%where $k\geq 1$,
\item (type III)
$$Kras_{\Delta_{p(v_j),m(j)}}^B(K_{v_j})\models C_k([[\Psi_j(\bar a,\pi_{m(j),p(v_j)}(\bar a))]]),$$
%where $s\in S$; for every $s$, $\psi_s(\bar x)$ is a formula of the language of hyperrings; 
%and $s(v)\geq 1$ 
%$$\prod_{v\notin S} H_{\Delta_0,v}^{P(v)} \models \neg C_k([[\psi(\pi(\bar a))]])\wedge$$ 
%$$(\prod_{v\notin S} H_{\Delta_0,v}^{P(v)} \models \neg C_{k-1}([[\psi(\pi(\bar a))]]) 
%\vee \prod_{v\in S} K_v \models \neg C_1([[\sigma(\bar a)]])) \wedge \dots \wedge$$
%$$(\prod_{v\notin S} H_{\Delta_0,v}^{P(v)}\models \neg C_{1}([[\psi(\pi(\bar a))]]) \vee 
%\prod_{v\in S} K_v \models \neg C_{k-1}([[\sigma(\bar a)]])) \wedge$$ 
%$$\prod_{v\in S} K_v \models \neg C_k([[\sigma(\bar a)]]),$$
\end{itemize}
where $\Psi(\bar x,\bar y),\Psi'(\bar x,\bar y),\Psi_1(\bar x,\bar y),\dots,\Psi_s(\bar x,\bar y)$ 
are formulas of the 2-sorted Krasner-Basarab language which are quantifier free in the field sort, 
$k\geq 1$, $j\in \{1,\dots,s\}$, $m(j)\geq 1$ (a positive 
integer depending on $v_j$), and $p(v_j)$ is the residue characteristic of $K_{v_j}$. 
The Boolean operations, 
$Fin$, $C_k$, and the Boolean value are expressible in the language of Basarab-Kuhlmann in each of the sorts. 
%\Leftrightarrow \exists j \in \{0,\dots,k\} \prod_{v\in S} K_v \models 
%\varphi(a(v)) \wedge \prod_{v\notin S}^{P} H_{\Delta_0,v}\models \Theta([[\phi(\pi(\bar a))]]).$$
\end{thm}
\begin{proof} 
%For an $\cL_{ring}$-formula $\psi(\bar x)$ and $\bar a\in \A_K^{fin}$, 
%suppose $\A_K^{fin}\models \psi(\bar a)$. Then applying Theorems ? and ?? 
%and Lemma ???, 
By Theorem $2_{\mathrm{sort}}$, for every $\bar a$ from $\A_K^{fin}$, the condition
$$
\A_K^{fin}\models \Phi(\bar a)
$$
is equivalent to a Boolean combination of conditions of the form
\begin{equation}\label{c}
\A_K^{fin}\models \Theta([[\Xi(\bar a)]])
\end{equation}
where $\Theta$ is either $Fin$ or $C_k$, and $\Xi(x)$ is a formula of the language of rings 
(resp.\ language of Basarab-Kuhlmann). 
By Theorem \ref{thm-kb}, there is a finite set 
$$S=\{v_1,\dots,v_s\}$$
of normalized non-Archimedean absolute values of $K$, 
positive integers $m(j)$ for $j\in \{1,\dots,s\}$, and 
formulas
$$\Psi(\bar x,\bar y),\Psi_1(\bar x,\bar y),\dots,\Psi_s(\bar x,\bar y)$$ 
from the $2$-sorted Krasner-Basarab language with no quantifiers of 
the field sort, and involving extra variables $\bar y$ from the sorts other than the field sort, 
such that for every $v\notin S$, and every $\bar a$ from $K_v$
$$K_v \models \Xi(\bar a)\Leftrightarrow Kras^B_{\Delta_{0,v}}(K_v)\models \Psi(\bar a,\pi_{0,v}(\bar a)),$$
and for all $j\leq s$, and every $\bar a$ from $K_{v_j}$
$$K_{v_j}\models \Xi(\bar a)\Leftrightarrow Kras^B_{\Delta_{p(v_j),m(j)}}(K_{v_j})\models 
\Psi_j(\bar a,\pi_{m(j),v_j}(\bar a)).$$
Here $p(v_j)$ is the residue characteristic of $K_{v_j}$. Note that
$$Fin(\{v: K_v\models \Xi(\bar x)\}) \Leftrightarrow Fin(\{v \notin S: K_v \models \Xi(\bar x)\})$$
$$\Leftrightarrow Fin(\{v: Kras^B_{\Delta_{0,v}}(K_v)\models 
\Psi(\bar x,\pi_{0,v}(\bar x))\}),$$
since $S$ is finite,
%and 
%$$C_k(\{v: K_v\models \varphi(x)\}) \Leftrightarrow C_k(\{v: H_{\Delta_0,v}\models \psi(\pi(x))\}).$$
%Using Lemma \ref{sigma-elim} replace the quantifier-free formulas from the $\cL_{ring}$-sort in the formulas 
%$\Psi(\bar x,\bar y),\Psi'(\bar x,\bar y),\Psi_1(\bar x,\bar y),\dots,\Psi_s(\bar x,\bar y)$ 
%with equivalent formulas from the language of hyperrings. 

Thus if $\Theta$ is $Fin$, then condition \ref{c} is equivalent to 
$$Kras_{S,0}^{\A}\models Fin([[\Psi(\bar a,\pi_{0,v}(\bar a))]]).$$
If $\Theta$ is $C_k$, for some $k\geq 1$, then condition \ref{c} is equivalent to 
%We now consider the case when $\Theta$ is $C_k$ or $\neg C_k$, for some $k\geq 1$. 
%We may assume that $k\geq 1$, otherwise 
%we have the equivalent condition $[[x\neq x]]$. 
%So the number of normalized non-archimedean absolute values $v$ 
%such that $K_v\models \varphi(x(v))$ is at least $0$ and strictly less than $k$. We may assume that this number is at least $1$ since 
%then we have the equivalent condition 
$$Kras_{S,0}^{\A}\models C_k([[\Psi(\bar a,\pi_{0,v}(\bar a))]])) \vee 
(Kras_{S,0}^{\A}\models C_{k-1}([[\Psi(\bar a,\pi_{0,v}(\bar a))]]) \wedge$$ 
$$\bigvee_{1\leq j\leq s} 
(Kras_{\Delta_{p(v_j),m(j)}}(K_{v_j})\models \Psi_{j}(\bar a,\pi_{m(j),p(v_j)}(\bar a))) \vee$$
$$\vdots$$
$$\vee (Kras_{S,0}^{\A}\models C_{k-t}([[\Psi(\bar a,\pi_{0,v}(\bar a))]]) \wedge $$
$$\bigvee_{(j_1,\dots,j_t)\in S^t} (Kras_{\Delta_{p(v_{j_1}),m(j_1)}}(K_{v_{j_1}}) \models 
\Psi_{j_1}(\bar a,\pi_{m(j_1),p(v_{j_1})}(\bar a))$$
$$\wedge \dots \wedge 
Kras_{\Delta_{p(v_{j_t}),m(j_t)}}(K_{v_{j_t}})\models \Psi_{j_t}(\bar a,\pi_{m(j_t),p(v_{j_t})}(\bar a))) \vee$$
$$\vdots$$
$$\vee \bigvee_{(j_1,\dots,j_k)\in S^k} 
(Kras_{\Delta_{p(v_{j_1}),m(j_1)}}(K_{v_{j_1}})\models \Psi_{j_1}(\bar a,\pi_{m(j_1),p(v_{j_1})}(\bar a))
\wedge \dots \wedge$$
$$Kras_{\Delta_{p(v_{j_k}),m(j_k)}}(K_{v_{j_k}})\models \Psi_{j_k}(\bar a,\pi_{m(j_k),p(v_{j_k})}(\bar a)))$$
%\bigvee_{(j_1,\dots,j_k)\in S^k} K_{v_j}/1+\cM^{m_j}\models \psi_j(\pi_{v,m_j}(\bar a)).$$
This yields (II). Taking negations, we deduce (III) for the case $\Theta$ is $\neg C_k$, for some $k$.
The last statement concerning the definitions of $Fin$, Boolean structure, and Boolean value in each sort of the 
Krasner-Basarab language follow from the remarks before the theorem.
\end{proof}

\section{Stable embedding}

It has been interesting and important in recent years, in the general area of valued fields, 
to analyze {\it stable embeddings and 
interpretations} (cf.\ \cite{HHM}). 
For example for Henselian fields coming under the Ax-Kochen-Ershov analysis, one knows that 
parametrically definable subsets of the value group, using the three sorted language in (9) of Section \ref{sec-sorts}, 
are already parametrically definable in the value group (cf.\ \cite{HHM}).

We show now that the local fields $K_v$ are stably embedded in the adeles $\A_K$ (via the identification of 
$K_v$ with $e_{\{v\}}K_v$). 
\begin{thm} Let $X$ be a definable subset of $\A_K^n$, where $n\geq 1$. 
Let $e$ be a minimal idempotent. 
Then $X\cap (e\A_K)^n$ is definable in $e\A_K$.
\end{thm}
\begin{proof} By Corollary \ref{adelic-qe},  
$X$ is defined by a Boolean combination of formulas of the type:

$$Fin([[\Psi(x_1,\dots,x_n,\bar a)]]),$$
$$C_j([[\Phi(x_1,\dots,x_n,\bar a)]]),$$
where $j\geq 1$, $\bar a\in \A_K^m$, 
and $\Psi$ and $\Phi$ are from the ring language. 
Recall that $$[[\Phi(x_1,\dots,x_n,\bar a)]]=1$$ is equivalent to 
$\neg C_1([[\neg \Phi(x_1,\dots,x_n,\bar a)]])$.

%For $\beta_1,\dots,\beta_n\in K_v$, 
%$$(\beta_1,\dots,\beta_n) \in X\cap (e\A_K)^n$$ 
%if and only if
%$$(f_1,\dots,f_n)\in X,$$
%where $f_j(v)=\beta_j$ and $f_j(w)=0$ for $w\neq v$. 
Let $b_1,\dots,b_n\in e\A_K$. Let $v\in V_K$ be the normalized valuation that corresponds to $e$. Then 
$[[\Phi(b_1,\dots,b_n,\bar a)]]$ is equal to
$$\{w\neq v: K_w\models \Phi(0,\dots,0,\bar a(w))\} \cup \{v\}$$
if
$$K_v\models \Phi(b_1(v),\dots,b_n(v),\bar a(v)),$$
and to
$$\{w\neq v: K_w\models \Phi(0,\dots,0,\bar a(w))\}$$
if
$$K_v\models \neg \Phi(b_1,\dots,b_n,\bar a(v))\}.$$ 
Now let
$$f_1,\dots,f_n\in X\cap (e\A_K)^n.$$
We have two cases to consider:

\

{\bf The case of $Fin$}: 

\

In this case we have 
%(1) $[[\Phi(f_1,\dots,f_n,\bar a)=1$ if and only if 
%$$K_v\models \Phi(\beta_1,\dots,\beta_n,\bar a(v))$$
%and for all $w\neq v$, 
%$$K_w\models \Phi(0,\dots,0,\bar a(w)).$$
$$\A_K\models Fin([[\Phi(f_1,\dots,f_n,\bar a)]])
\Leftrightarrow 
\A_K\models Fin(\Phi(0,\dots,0,\bar a)]]).$$
There are two sub-cases.

(i) $\A_K\models Fin([[\Phi(0,\dots,0,\bar a)]])$. In this case 
$$X\cap (e\A_K)^n=(e\A_K)^n$$
%$$\{w\neq v: K_w\models \Phi(0,\dots,0,\bar a(w)\}$$
%is finite.
(ii) $\A_K\models \neg Fin([[\Phi(0,\dots,0,\bar a)]])$. In this case,
$$X\cap (e\A_K)^n=\emptyset.$$

\

{\bf The case of $C_j$}: 

\

Note that 
%(3) $C_j([[\Phi(f_1,\dots,f_n,\bar a)]])$ holds if and only if 
%either
%%nd 
%$$\{w\neq v: K_w\models \Phi(0,\dots,0,\bar a(w))\}$$
%has cardinal at least $j-1$; or 
%$$K_v\models \neg \Phi(\beta_1,\dots,\beta_n,\bar a(v))$$
%and 
%$$\{w\neq v: K_w\models \Phi(0,\dots,0,\bar a(w))\}$$
%has cardinal at least $j$. 
%Note that given $\Phi$, the side conditions about what happens at $w\neq v$ 
%are independent of $v$. It follows that $X\cap \A_K^n$  is defined, for 
%each of the finitely many truth configurations of side conditions, by a Boolean 
%combination of conditions
%$$\Phi(x_1,\dots,x_n,\bar a(v)),$$
%which completes the proof.
$$\A_K\models C_j([[\Phi(f_1,\dots,f_n,\bar a)]])$$
if and only if either
$$\A_K\models (C_{j-1}([[\Phi(f_1,\dots,f_n,\bar a)]])\wedge \neg C_{j}([[\Phi(0,\dots,0,\bar a)]])$$
and
$$e\A_K\models \Phi(f_1(v),\dots,f_n(v),\bar a(v)),$$
or 
$$\A_K\models C_j([[\Phi(0,\dots,0,\bar a)]])$$
(and in this case there is no condition on $e\A_K$).

In the first case
$$X\cap (e\A_K)^n=\{(g_1,\dots,g_n)\in (e\A_K)^n: \Phi(g_1,\dots,g_n,\bar a)\},$$
and in the second case
$$X\cap (e\A_K)^n=(e\A_K)^n.$$
In all the cases, it is thus clear that $X\cap (e\A_K)^n$ is definable with parameters from $e\A_K$.
\end{proof}

\begin{note} Although we do not take the time to state a general result here, it is clear 
that one has for generalized products a very general stable embedding theorem for factors. 
\end{note}
%\section{The value group monoid of $\A_{K}^f$}
%We refer to Section ? for the discussion of
Recall from Section \ref{val-monoid} that we have the product valuation $\prod v$ from $\A_K^{fin}$ to the 
restricted product $\Gamma$ of the lattice-ordered monoids $\Z\cup\{\infty\}$. Evidently 
$\Gamma$ is interpretable in the ring $\A_K^{fin}$.

\begin{thm} The value monoid $\Gamma$ of $\A_{\Q}^{fin}$ is not stably interpreted (via the valuation map).
\end{thm}
\begin{proof}
We can define in $\A_{\Q}$ the set $X$ of 
idempotents whose support is the set of minimal idempotents $e$ whose corresponding prime 
$p$ is congruent to $1$ modulo $4$. Indeed, let $\Psi$ be a sentence that holds in $\Q_p$ for exactly 
the primes $p$ with $p\equiv 1(\mathrm{mod}~4)$, and let $\Psi'$ be a sentence that holds in all 
non-Archimedean local fields and fails in all the Archimedean local fields. Then
$$X=\{x\in \A_{\Q}: supp(x)=[[\Psi \wedge \Psi']]\}.$$

The image of $X$ under the product valuation $\prod v$ 
is the set $Y$ of all $g$ in $\prod_p (\Z\cup \{\infty\})$ which are $0$ at $p$ 
and $\infty$ elsewhere. It is an easy exercise using the Feferman-Vaught Theorem (or Theorem $2_{\mathrm{sort}}$) 
applied to $\Gamma$ 
to show that this set is not definable in the value monoid, 
using an appropriate modification of the Pressburger elimination in the factors.
\end{proof}

\section{A remark about $NTP_2$}
The property of not having the tree property of the second kind $NTP_2$ is a generalization of the properties of 
simple and NIP (the negation of the independence property). It is known 
that ultraproducts of $\Q_p$ and certain valued difference fields 
have $NTP_2$ (cf.\ \cite{CH}).  

Fix a number field $K$. The theory of $\A_K$ has the independence property in two different ways. The first is via the 
residue fields (appealing to Duret \cite{Duret}, also cf.\ \cite{FJ}), 
and the other comes from the definable Boolean algebra $\B_K$. 
Now the problem with the residue fields does not extend to the property $NTP_2$. 
However, it turns out that we still have the following.
\begin{thm}\label{thm-ntp} The theory of finite adeles $\A_K^{fin}$ and the theory of adeles 
$\A_K$ do not have the property $NTP_2$.
\end{thm}
\begin{proof} To show the negation of $NTP_2$ we have to produce a formula $\Psi(x,y)$, and an array
$$a_{11},a_{12},\dots$$
$$a_{21},a_{22},\dots$$ 
$$\vdots$$
so that for fixed $j$,
$$\{\Psi(x,a_{jk}): k\geq 1\}$$
is inconsistent, but for each $f:\N\longrightarrow \N$,
$$\{\Psi(x,a_{jf(j)}): j\geq 1\}$$
{\it is} consistent. ($x$ and $y$ can be tuples).

Firstly, put $a_{jk}=a_{j1}^k$, $k\geq 1$. Secondly, for each $j$ 
pick a minimal idempotent $e_j$ so that $e_j\A_K^{fin}$ is the $K_v$, $v\in V_K^f$, which has 
residue field of 
characteristic $p_j$, the $j$th prime. Finally, pick $a_{j1}$ to be an atom 
such that the coordinate at which it is nonzero lies in the maximal ideal of the valuation ring 
of $e_j\A_K$.
%Now let $x$ be a 1-tuple, $y$ a 2-tuple $(y_1,y_2)$, and let $\Psi(x,y_1,y_2)$ ''say`` that $y_2$ is an atom, the support of 
%$y_1$ is $\{y_2\}$, $y_1$ is in the maximal ideal for the valuation ring over $y_2$, and $xy_2$ and $y_1$ have the same value 
%over $y_2$. There is such a ring formula by the results of \cite{CDLM} on uniform definability of the valuation for all $K_v$. 

Now take the formula $\Psi(x,y)$ to be
$$C_1([[y\neq 0]])\wedge \neg C_2([[y\neq 0]])\wedge [[y\neq 0]]=[[\rho(y)]]\leq [[\sigma(x,y)]],$$
where $\rho(y)$ is a formula of the language of rings which is equivalent in all the $K_v$ to the statement that 
$y$ has positive valuation, and 
$\sigma(x,y)$ is a formula of the language of rings which is equivalent in all the $K_v$ to the statement 
that $x$ and $y$ have the same valuation. Note that such formulas exist by the 
results in \cite{CDLM} on uniform definability of the valuation for all $K_v$ in the ring language. Here $v$ ranges 
over all non-archimedean valuations of $K$. Note that the first two conjuncts from the left state that 
$[[y\neq 0]]$ is a minimal idempotent, and the formula states that the support of $y$ is minimal and the nonzero 
coordinate of $y$ lies 
in the maximal ideal of the valuation ring, and $x$ and $y$ have the same valuation at that coordinate.
%So now let $A_{jk}=(a_{jk},e_j)$. Then for each $j$

Now it is clear that 
$$\{\Psi(x,a_{jk}): k\geq 1\}$$
is inconsistent, since for $k_1\neq k_2$
$$v(e_{j}a_{jk_1})\neq v(e_j(a_{jk_2}))$$
for $v$ the (normalized) valuation of $e_j\A_K^{fin}$ (since $k_1v(e_j(a_{j1}))\neq k_2v(e_j(a_{j1}))$).

However, for any $f:\N\longrightarrow \N$
$$\{\Psi(x,a_{j,f(j)}): j\geq 1\}$$
is consistent by choosing a finite adele $A\in \A_K^{fin}$ such that its coordinate $A(j)$ in $e_j\A_K^{fin}$ satisfies 
$$v(A(j))=f(j)v(a_{j,1}),$$
for all $j$. Note that there is such an element in $\A_K^{fin}$. 

This proves that $\A_K^{fin}$ does not have the property $NTP_2$. To deduce that the adeles $\A_K$ does not have 
$NTP_2$, it suffices to show that $\A_K^{fin}$ is definable in $\A_K$ (in the language of rings). To see this, 
take a sentence $\Theta$ which holds in all archimedean completions $K_v$ of $K$ but is not true in all the 
non-archimedean completions. Then $\A_K^{fin}$ can be defined as the set of all $f\in \A_K$ such that 
$f(v)=0$ for all $v\in [[\Theta]]$.
\end{proof}
\begin{note} The formula uniformly defining the valuation of all the local fields 
from \cite{CDLM} is existential-universal (this is shown in \cite{CDLM} to be optimal, i.e. there is no 
uniform universal or existential definition). 
It follows that the formula $\Psi(x,y)$ in the proof of Theorem \ref{thm-ntp} is universal-existential-universal.
\end{note}

%\bibliographystyle{amsplain}
%\bibliography{suppbib}

\begin{thebibliography}{10}

\bibitem{basarab}
{\c{S}}erban~A. Basarab, \emph{Relative elimination of quantifiers for
  {H}enselian valued fields}, Ann. Pure Appl. Logic \textbf{53} (1991), no.~1,
  51--74. \MR{1114178 (92j:03028)}

\bibitem{belair}
Luc B{\'e}lair, \emph{Substructures and uniform elimination for {$p$}-adic
  fields}, Ann. Pure Appl. Logic \textbf{39} (1988), no.~1, 1--17. \MR{949753
  (89j:03026)}

\bibitem{Cassels}
J.~W.~S. Cassels, \emph{Global fields}, Algebraic {N}umber {T}heory ({P}roc.
  {I}nstructional {C}onf., {B}righton, 1965), Thompson, Washington, D.C., 1967,
  pp.~42--84. \MR{0222054 (36 \#5106)}

\bibitem{CH}
Artem Chernikov and Martin Hils, \emph{Valued difference fields and
  ${N}{T}{P}_{2}$}, arXiv:12081341.

\bibitem{CDLM}
Raf Cluckers, Jamshid Derakhshan, Eva Leenknegt, and Angus Macintyre,
  \emph{Uniformly defining valuation in henselian valued fields with finite or
  pseudofinite residue field}, Annals of Pure and Applied Logic, to appear.

\bibitem{CL}
Raf Cluckers and Fran{\c{c}}ois Loeser, \emph{Constructible motivic functions
  and motivic integration}, Invent. Math. \textbf{173} (2008), no.~1, 23--121.
  \MR{2403394 (2009g:14018)}

\bibitem{CC}
Alain Connes and Caterina Consani, \emph{The hyperring of ad\`ele classes}, J.
  Number Theory \textbf{131} (2011), no.~2, 159--194. \MR{2736850
  (2012i:14004)}

\bibitem{DL}
Jan Denef and Fran{\c{c}}ois Loeser, \emph{Definable sets, motives and
  {$p$}-adic integrals}, J. Amer. Math. Soc. \textbf{14} (2001), no.~2,
  429--469 (electronic). \MR{1815218 (2002k:14033)}

\bibitem{DM-boole}
Jamshid Derakhshan and Angus Macintyre, \emph{Enrichments of boolean algebras:
  a unifying treatment of some classical and some novel examples}, Fundamenta
  Mathematicae, to appear.

\bibitem{DM-ad}
Jamshid Derakhshan and Angus Macintyre, \emph{Model theory of adeles {I}}, In
  preperation.

\bibitem{Duret}
Jean-Louis Duret, \emph{Les corps faiblement alg\'ebriquement clos non
  s\'eparablement clos ont la propri\'et\'e d'ind\'ependence}, Model theory of
  algebra and arithmetic ({P}roc. {C}onf., {K}arpacz, 1979), Lecture Notes in
  Math., vol. 834, Springer, Berlin, 1980, pp.~136--162. \MR{606784
  (83i:12024)}

\bibitem{FV}
S.~Feferman and R.~L. Vaught, \emph{The first order properties of products of
  algebraic systems}, Fund. Math. \textbf{47} (1959), 57--103. \MR{0108455 (21
  \#7171)}

\bibitem{feferman}
Solomon Feferman, \emph{Lectures on proof theory}, Proceedings of the {S}ummer
  {S}chool in {L}ogic ({L}eeds, 1967) (Berlin), Springer, 1968, pp.~1--107.
  \MR{0235996 (38 \#4294)}

\bibitem{FJ}
M.~Fried and M.~Jarden, \emph{Field arithmetic}, vol.~3, Ergebnisse der
  Mathematik und ihrer Grenzgebiete, no.~11, Springer, 1986.

\bibitem{HHM}
Deirdre Haskell, Ehud Hrushovski, and Dugald Macpherson, \emph{Stable
  domination and independence in algebraically closed valued fields}, Lecture
  Notes in Logic, vol.~30, Association for Symbolic Logic, Chicago, IL, 2008.
  \MR{2369946 (2010c:03002)}

\bibitem{HK}
Ehud Hrushovski and David Kazhdan, \emph{Integration in valued fields},
  Algebraic geometry and number theory, Progr. Math., vol. 253, Birkh\"auser
  Boston, Boston, MA, 2006, pp.~261--405.

\bibitem{kochen}
Simon Kochen, \emph{Ultraproducts in the theory of models}, Ann. of Math. (2)
  \textbf{74} (1961), 221--261. \MR{0138548 (25 \#1992)}

\bibitem{krasner}
Marc Krasner, \emph{A class of hyperrings and hyperfields}, Internat. J. Math.
  Math. Sci. \textbf{6} (1983), no.~2, 307--311. \MR{701303 (84f:16042)}

\bibitem{KK}
G.~Kreisel and J.-L. Krivine, \emph{Elements of mathematical logic. {M}odel
  theory}, Studies in Logic and the Foundations of Mathematics, North-Holland
  Publishing Co., Amsterdam, 1967. \MR{0219380 (36 \#2463)}

\bibitem{kuhlmann}
Franz-Viktor Kuhlmann, \emph{Quantifier elimination for {H}enselian fields
  relative to additive and multiplicative congruences}, Israel J. Math.
  \textbf{85} (1994), no.~1-3, 277--306. \MR{1264348 (95d:12012)}

\bibitem{MR}
Michael Makkai and Gonzalo~E. Reyes, \emph{First order categorical logic},
  Lecture Notes in Mathematics, Vol. 611, Springer-Verlag, Berlin, 1977,
  Model-theoretical methods in the theory of topoi and related categories.
  \MR{0505486 (58 \#21600)}

\bibitem{Pas}
Johan Pas, \emph{Uniform {$p$}-adic cell decomposition and local zeta
  functions}, J. Reine Angew. Math. \textbf{399} (1989), 137--172. \MR{1004136
  (91g:11142)}

\bibitem{pillay-book}
Anand Pillay, \emph{Geometric stability theory}, Oxford Logic Guides, vol.~32,
  The Clarendon Press Oxford University Press, New York, 1996, Oxford Science
  Publications. \MR{1429864 (98a:03049)}

\bibitem{serre}
Jean-Pierre Serre, \emph{Propri\'et\'es conjecturales des groupes de {G}alois
  motiviques et des repr\'esentations {$l$}-adiques}, Motives ({S}eattle, {WA},
  1991), Proc. Sympos. Pure Math., vol.~55, Amer. Math. Soc., Providence, RI,
  1994, pp.~377--400. \MR{1265537 (95m:11059)}

\bibitem{weispfenning-hab}
V.~Weispfenning, \emph{Model theory of lattice products.}, Habilitation,
  Universitat Heidelberg (1978).

\end{thebibliography}
\end{document}